\documentclass[11pt,a4paper]{amsart}
\usepackage{geometry, empheq}
\usepackage[english]{babel}
\usepackage{amsmath}
\usepackage{amssymb}
\usepackage{amsthm}
\usepackage{xcolor,graphicx}
\usepackage{color}
\usepackage{varioref}
\usepackage{nicefrac}
\usepackage{caption}
\usepackage{subcaption}
\usepackage{float}
\usepackage[normalem]{ulem}
\usepackage{multirow}
\usepackage{enumerate}
\usepackage[colorinlistoftodos]{todonotes}
\usepackage{amsfonts}
\usepackage{booktabs}
\usepackage{siunitx}
\usepackage{tikz}
\usepackage{pdflscape}
\usepackage{pgfplots}
\usepgfplotslibrary{external} 
\tikzset{external/mode=graphics if exists}
\usetikzlibrary{external}
\usepackage{algorithm}
\usepackage{algorithmic}

\newcommand{\arrow}{\rightarrow}
\newcommand{\ds}{\displaystyle}

\newcommand{\kw}{\rule{2mm}{2mm}}

\newtheorem{example}{\textbf{Example}}
\newcommand{\norm}[1]{{\|  #1\|}}
\renewcommand{\l}{\left}
\renewcommand{\r}{\right}
\newcommand{\om}{\Omega}
\newcommand{\ga}{\Gamma}
\newcommand{\fsparse}{\Upsilon}
\newcommand{\reals}{\mathbb{R}}
\newcommand{\sign}{\mathrm{sign}}

\DeclareMathOperator{\argmin}{argmin}
\renewenvironment{proof}{{Proof.}}{\hfill\kw}

\newcounter{theassumption}
\newtheorem{assumption}[theassumption]{{Assumption}}
\newtheorem{proposition}{{Proposition}}
\newtheorem{definition}{{Definition}}
\newtheorem{lemma}{{Lemma}}
\newtheorem{remark}{{Remark}}
\newtheorem{corollary}{{Corollary}}
\newtheorem{theorem}{{Theorem}}

\begin{document}
	\title[FEM approximation for nonconvex PDE optimal control problems]{Numerical approximation of regularized non--convex elliptic optimal control problems by the finite element method}
	
	\author{Pedro Merino$^\ddag$ and Alexander Nenjer$^\ddag$}
	\address{$^\ddag$Research Center of Mathematical Modeling (MODEMAT) and Department of Mathematics, Escuela Polit\'ecnica Nacional, Quito, Ecuador}
	\thanks{$^*$This research has been supported by Research Project PIGR-19-03 funded by Escuela Polit\'ecnica Nacional, Quito--Ecuador.}
	
	\smallskip
	\begin{abstract}
		We investigate the numerical approximation of an elliptic optimal control problem which involves a nonconvex local regularization of the $L^q$-quasinorm penalization (with $q\in(0,1)$) in the cost function. Our approach is based on the \emph{difference--of-- convex} function formulation, which leads to first-order necessary optimality conditions, which can be regarded as the optimality system of an auxiliar convex $L^1$--penalized optimal control problem. 
		We consider piecewise--constant finite element approximation for the controls, whereas the state equation is approximated using piecewise--linear basis functions. Then, convergence results are obtained for the proposed approximation. Under certain conditions on the support's boundary of the optimal control, we deduce an order of  $h^\frac12$ approximation rate of convergence where $h$ is the associated discretization parameter. We illustrate our theoretical findings with numerical experiments that show the convergence behavior of the numerical approximation.%
	\end{abstract}
	
	\maketitle
	
	\section{Introduction}
	Optimal control problems with nonconvex $L^q$-quasinorms with $q \in (0,1)$ can be regarded as an approximation of the so-called $L^0$ sparse penalized optimal control problems, see \cite{itoku2014} and \cite{merino2019}. This kind of penalization induces attractive sparsification effects on the control, which might be useful in applications requiring a localized action of the control within the domain e.g., inverse sparse reconstruction problems. 
	
	Optimal controls induced by $L^q$-quasinorm penalizers share similarities with those promoted by $L^1$-norm. It is well known that these controls tend to have small supports, depending on the sparse regularization cost, see \cite{stadler06}. However, a simple but essential difference in considering $L^q$-quasinorms is that optimal controls might jump at the boundary of its support. This feature is potentially advantageous for specific applications. In particular, in those applications where sparse solutions are crucial in achieving the desired state, see \cite{casas2013etal}. Other examples arise in image-processing-related applications where jumps can be used to achieve certain graphical features. 
	
	One recognizable difficulty in considering the $L^q$-quasinorm in the cost function is the lack of convexity of the penalizer. Therefore,  $L^p$-spaces ($p>1$) are not well suited for this class of problems. In \cite{itoku2014}, a penalization in the gradient is added to the cost, which allows using compactness arguments to guarantee the existence of solutions in $H^1$. However, the gradient cost term eliminates the jumps of the solutions, which is an essential feature of the quasinorm penalization. Therefore, a more convenient space to promote piecewise smooth controls is the space of functions of bounded variation denoted by $BV$. In this case, the existence of solutions can also be argued by compactness arguments.

	Several contributions published in the last decade are devoted to the numerical approximation of optimal control problems related to sparse controls. A linear order of convergence was derived in \cite{wachsmuth} for linear elliptic optimal control problems with an $L^1$ cost term, where a general discretization scheme including piecewise constant approximation was studied. The same order of convergence was obtained for sparse optimal control problems governed by linear elliptic equations by considering regular Borel measures as control space, where controls were approximated using linear combinations of Dirac measures, see \cite{caclaku2012}, and \cite{pievex2013}. The semilinear case with $L^1$--norm penalization in the cost was studied in \cite{casherwas2012} and \cite{casherwas2012b} for piecewise constant and piecewise linear approximations of the control, respectively. These papers also reported a linear order of error for both approximations.

	On the other hand, the distributed optimal control problem with a $L^q$--quasinorm ($q\in (0,1)$) penalization has the form:
	
	\begin{equation}
		\tag{$P$} \label{eq:OCP}
		\begin{cases}
			\displaystyle\min_{(y,u)} ~\frac{1}{2}\| y-y_d \|^2_{L^2(\om)}+\frac{ \alpha}{2}\|u\|^2_{L^2(\om)}+\beta \int_\om |u|^{q} dx\\
			\hbox{ subject to: }\\
			\hspace{40pt}\begin{array}{rll}
				A y=&u + f, &\hbox{in  } \om, \\
				y=&0, &\hbox{on  }  \Gamma,
			\end{array}
			\\
			\hspace{45pt} u \in U_{ad} \cap BV(\om),
		\end{cases}
	\end{equation}
	in our setting, $\om \subset \reals^2$ is star-shaped domain with boundary $\Gamma$. $y_d$ is  given in $L^p(\om)$, with $p>n$ and $f$ is given in $L^2(\om)$.  The admissible control set $U_{ad}$ is a closed convex set in $L^2(\om)$. Let $A$ 
	be a uniformly elliptic second--order differential operator of the form:
	\begin{equation}\label{eq:A}
		(Ay) (x) = - \sum_{i,j=1}^n \frac{\partial}{\partial x_i}	\l( a_{ij}(x)\frac{\partial y(x)} {\partial x_j }\r) +c_0 y(x),
	\end{equation}
	with coefficients $a_{ij}$ in $C(\bar\om)$, and $c_0\geq 0$ in $L^{\infty}(\om)$. Moreover, the matrix $(a_{ij})$ is symmetric and fulfills the uniform ellipticity condition: \[ \exists\, \sigma>0 : \quad \ds  \sum_{i,j=1}^n a_{ij}(x) \xi_i\xi_j \geq \sigma |\xi|^2, \quad \forall \xi \in \reals^n, \text{for almost all } x \in \om. \]
	
	A related problem was studied in \cite{itoku2014} with a quadratic penalizer on the gradient. There, the authors established the existence of solutions of controls in $H_0^1(\om)$ using compactness arguments. Moreover, by using a suitable regularization of the $L^q$-quasinorm, a primal-dual scheme for its numerical solution was also proposed. 
	On the other hand, the optimal control problem with an additional quadratic $L^2$ penalizer on the control was discussed in \cite{merino2019} through a difference--of--convex function formulation of a regularization of the nonconvex problem. A regularizing assumption on the control space in $H^1$ was also considered due to the nonconvex nature of the problem and its intrinsic difficulties in proving the solution's existence for \eqref{eq:OCP}. 
	
	\emph{Our contribution.} 	This paper aims to analyze the numerical approximation by the finite element method of the regularized nonconvex optimal control problem  \eqref{eq:OCP}. Although several related papers investigated optimality conditions and numerical methods (see \cite{itoku2008},\cite{dwachsmuth}), we cannot track literature for a priori error estimates for the discretization by the finite element method for this type of problem. 
	
	One of the difficulties in studying the numerical approximation of nonconvex problems is related to the expected nonuniqueness of solutions. For differentiable nonconvex problems, the associated analysis is based on second--order optimality conditions. See for example \cite{aracastro2002}.
	
	Moreover, we consider the space $BV(\om)$ of bounded variation functions as a control space that allows the presence of jumps in the solutions. We rely on the interpolation operator introduced in \cite{bartels15}. It is known that piecewise constant functions are not a suitable approximation for problems involving functions of bounded variation due to the norm in $BV(\om)$. However, the total variation term is not present in our problem.
	
	It turns out that the difference-of-convex function formulation, together with its associated optimality system derived in \cite{merino2019}, are useful to carry out an approximation analysis for the finite element approximation and help us to establish second-order sufficient conditions for local solutions of our problem. In fact, the DC formulation can be expressed as a linearly perturbed $L^1$ sparse optimal control, for which we can carry out a similar analysis as in \cite{wachsmuth}; then, we end up expressing the error in terms of the corresponding linear perturbations. However, the results of \cite{wachsmuth} cannot be applied to our problem. We still require to analyze this perturbation (which depends on the control), taking into account the numerical approximation for the control space. Therefore, for our analysis, we use the pointwise characterization of the solution resulting from the maximum principle established in \cite{itoku2008}. Following a similar assumption on the active sets as in \cite{wachsmuth}, we can obtain an error estimate of order $h^\frac12$ provided that the boundary of the solution support is contained in a subset of the mesh whose measure is proportional to the discretization parameter.

	\emph{Organization of the paper.} In Section \ref{s:problem}, we state the optimal control problem with nonconvex penalties and briefly discuss its main properties. The discretization technique is presented in Section 3. Section 4 is devoted to estimating the convergence rate of the FEM approximation. Finally, we confirm the theoretical findings with numerical measures of the convergence rate with respect to $h$.
	
	
	\section{Regularized optimal control problem}\label{s:problem}
	
	For $q\in(0,1)$ and $\gamma >>1$, we introduce the regularization mapping 
	
	\begin{equation}
		u \mapsto \fsparse_{q,\gamma} (u) := \int_\om h_{q,\gamma}(u(x))^{q} dx,
	\end{equation}
	where $h_{q,\gamma}$ is Huber-like local smoothing of the absolute value introduced in \cite{merino2019}, and defined by
	$$
	h_{q,\gamma}(t):=
	\left\{ \begin{array}{ll}
		\ds q{\gamma^{\frac{1-q}{q}}}|t|^{\frac{1}{q}},& \hbox{if } t \in[-\frac{1}{\gamma},\frac{1}{\gamma}],\vspace{2mm}
		\\
		\ds|t| - \frac{1-q}{\gamma},&   \text{ otherwise. }
	\end{array} \right.
	$$ 
	The mapping $\fsparse_{p,\gamma}$ allows us to define the following family of (non--convex and non--differentiable) optimal control problems that approximate problem \eqref{eq:OCP}:
	\begin{equation}
		\tag{$P_\gamma$} \label{e:OCP_c}
		\begin{cases}
			\displaystyle\min_{(y,u)}\mathcal{J}_\gamma(u,y):= ~\frac{1}{2}\| y-y_d \|^2_{L^2(\om)}+\frac{\alpha}{2}\| u\|^2_{L^2(\om)}+\beta \fsparse_{q,\gamma}(u)\\
			\hbox{ subject to: }\\
			u \in U_{ad}  \cap BV(\om)\qquad\text{  and}
			\hspace{20pt}\begin{array}{rll}
				A y=&u + f, &\hbox{in  } \om, \\
				y=&0, &\hbox{on  }  \Gamma.
			\end{array}
		\end{cases}
	\end{equation}
	Here, the admissible of controls $U_{ad}$ corresponds to the box--constraint type, i.e.
	\begin{equation}\label{eq:U_ad}
		U_{ad}=\{u\in L^2(\om): u_a \leq u(x) \leq u_b, \, \text{a.a. } x\in \om \},
	\end{equation}  
	for reals  $u_a$ and $u_b$, such that $u_a<0<u_b$.

	By the classical theory of elliptic partial differential equations,  the state equation is formulated in the weak sense by introducing the bilinear form associated with the elliptic operator \eqref{eq:A}, denoted by $a$. 
	
	Thanks to the Lax--Milgram theorem and elliptic regularity \cite{brezis2010}, we know that for every $w \in L^2(\om)$ there exist $y\in H_0^1(\om)$ and a positive constant $c_a>0$, satisfying: 
	\begin{align}\label{eq:state}
		&a(y,v) = (w,v)_{L^2(\om)}, \quad \forall v \in H_0^1(\om),\\
		&\norm{y}_{H_0^1(\om)} \leq c_a\norm{w}_{L^2(\om)}.
	\end{align}
	Let us introduce the solution operator  $S:L^2(\om ) \arrow H_0^1(\om)$ corresponding to the linear and continuous operator which assigns to each $u \in L^2(\om)$ the corresponding solution $y=y(u) \in H_0^1(\om)$ satisfying the state equation in \eqref{e:OCP_c}. Moreover, as in \cite{wachsmuth}, the operator $S$ and its adjoint $S^*$ are continuous from $L^2(\om)$ to $L^\infty(\om)$. Thus, the state $y$ associated to the control $u$ has the representation  $y=S(u + f)=Su+y_f$, with $y_f=Sf$. 
	
	\begin{assumption}\label{a:cotaDu}
		There exists a positive constant $M$ such that $\norm{\nabla u}_{\mathcal{M}(\om)} \leq M$ for all $u\in U_{ad}$. Here $\norm{\cdot}_{\mathcal{M}(\om)}$ is the norm on the space of regular Borel measures.
	\end{assumption}
	
	\begin{theorem} \label{t:exist}
		Under Assumption \ref{a:cotaDu}, there exists a solution for problem \eqref{e:OCP_c}
	\end{theorem}
	\begin{proof}
		Let $\{(y_k,u_k)\}_{k\in \mathbb{N}} \subset H_0^1(\om)\times U_{ad}$ a minimizing sequence 	for problem \eqref{e:OCP_c} such that $y_k \in H_0^1(\om)$ is the corresponding state associated to $u_k$. Let us check the boundedness of $\{(y_k,u_k)\}_{k\in \mathbb{N}}$. Indeed, $u_k \in U_{ad}$ hence $\norm{u_k}_{L^\infty (\om)} \leq \max \{-u_a,u_b\}$. Furthermore, by Assumption \ref{a:cotaDu}  the sequence $\{u_k\}_{k\in \mathbb{N}}$ is bounded in $BV(\om)$.  
		
		On the other hand, since $y_k$ satisfies the state equation
		\begin{equation*}
			\begin{array}{rll}
				A y_k=&u_k + f, &\hbox{in  } \om, \\
				y=&0, &\hbox{on  }  \Gamma
			\end{array}	
		\end{equation*}
		and, since $BV(\om) \hookrightarrow L^2(\om)$ in two dimensions (see \cite[Theorem 10..1.3]{attouch2014}), the Lax-Milgram Theorem guarantees the existence of a positive constant $c$, such that
		
		$$ \norm{y_k}_{H_0^1(\om)} \leq c \norm{u_k}_{L^2(\om)} \leq c \max \{-u_a,u_b\},$$ 
		which shows that the sequence $\{y_k\}_{k\in \mathbb{N}}$ is bounded in $H_0^1(\om)$. Thus, there is a subsequence (without renaming) $\{(y_k,u_k)\}_{k\in \mathbb{N}}$, such that 
		$u_k \overset{\ast}{\rightharpoonup} \bar u $ in $BV(\om)$ and $y_k \rightharpoonup \bar y$ in $H_0^1(\om)$. We have that $u_k \arrow \bar u$ in $L^1(\om)$  and $y_k \arrow \bar y$ in $L^2(\om)$. Since $U_{ad}$ is closed, it also follows that $u \in U_{ad}$.
		
		By continuity of $\fsparse_{q,\gamma}$ from $L^1(\om)$ into $\reals$, see \cite[Lemma 2]{merino2019} we can pass to the limit to obtain that
		\begin{align*}
			\lim_{k\arrow \infty} \fsparse_{q,\gamma}(u_k) = \fsparse_{q,\gamma}(\bar u) 
			\quad\text{ and }\quad\lim_{k\arrow \infty} \frac{1}{2}\| y_k-y_d \|^2_{L^2(\om)} = \frac{1}{2}\| \bar y-y_d \|^2_{L^2(\om)}. 
		\end{align*}
		Moreover, $u \mapsto \int_\om u^2 dx$ is lower semi-continuous. Since $\{u_k\}_{k\in\mathbb N}\subset L^2(\om)$, then $u_k \arrow \bar u$ in $L^1(\om)$ implies
		\begin{align*}
			\liminf_{k\arrow\infty}\frac{\alpha}{2}\|  u_k\|^2_{L^2(\om)}=\liminf_{k\arrow\infty}\frac{\alpha}{2} \int_\om  u_k(x)^2 dx   \geq \frac{\alpha}{2}\|  \bar u\|^2_{L^2(\om)}. 
		\end{align*}
		Altogether, we infer that $J_\gamma (\bar y,\bar u)\geq \lim_{k \arrow \infty} J_\gamma (y_k,u_k)= \inf J_\gamma (y,u)$. Thus $(\bar y,\bar u)$ is a solution for \eqref{e:OCP_c}.
	\end{proof}

	\subsection{Optimality conditions via DC--programming} DC--functions consist of those functions represented by the difference of two convex functions. This class of functions originated the DC--programming theory, which is well known in nonconvex optimization cf. \cite{hiriart1989}. A \emph{difference--of--convex functions} formulation for elliptic optimal control problems involving $L^q$--quasinorms was applied in \cite{merino2019}. This formulation can be conveniently analyzed in the framework of DC programming in order to derive optimality conditions in the form of a KKT system. By replacing this expression in the objective function and incorporating the indicator function ${I}_{U_{ad}}$ for the admissible control set, we get the following reduced problem:
	\begin{equation}
		\tag{$P'$} \label{eq:OCP'}
		\displaystyle\min_{u \in BV(\om)} J_\gamma(u):=~\frac{1}{2}\| Su + y_f-y_d \|^2_{L^2(\om)}+\frac{\alpha}{2}\|u\|^2_{L^2(\om)}+\beta \fsparse_{p,\gamma}(u) + {I}_{U_{ad}}.
	\end{equation}
	Let, $$F(u) := \frac12 \norm{Su +Sf-y_d}^2_{L^2(\om)} + \frac{\alpha}{2} \norm{u}_{L^2(\om)}^2.$$ Then, taking into account that $U_{ad}\subset BV(\om)\cap L^\infty (\om)$, the DC-formulation of $J_\gamma$ is defined by expressing $J_\gamma = G-H$, where 
	\begin{align}\label{eq:GH}
		\begin{array}{ll}
			&\begin{array}{lrlll}
				G: & L^2(\om) & \arrow &\reals \\
				& u		& \mapsto &  G(u) &: = F(u) + \beta \delta_\gamma \norm{u}_{L^1(\om)}+{I}_{U_{ad}}, \quad \text{ with } \delta_\gamma := q^q \gamma^{1-q}
			\end{array} \\
			&\begin{array}{lrll}
				H: & L^2(\om) & \arrow &\reals \\
				& u		& \mapsto &  H(u) : =  \beta \l( \delta_\gamma \norm{u}_{L^1(\om)} - \fsparse_{p,\gamma}(u) \r),
			\end{array}
		\end{array}
	\end{align}
	therefore, \eqref{eq:OCP'} admits the formulation 
	\begin{equation}
		\label{e:OPT-DC}\tag{DC}
		\displaystyle\min_{u} J_\gamma(u)= G(u) - H(u).
	\end{equation} 
	This representation permits the application of DC programming theory, see \cite{hiriart1989}. In fact, a local solution of the  problem \eqref{e:OPT-DC} must satisfy the so-called \textit{DC--criticality}: 
	\begin{equation*}
		\partial H(\bar u) \subset \partial G( \bar u),
	\end{equation*}	
	from which the associated optimality system is derived, see \cite[Theorem 6]{merino2019}. We summarize these conditions in the following proposition. 
	\begin{proposition}\label{t:fonc} Let $\bar u \in U_{ad}$ be a solution of \eqref{e:OCP_c}, then there exist:  $\bar y = S \bar u+y_f$ in $H_0^1(\om)$, with $y_f=Sf$  
		an adjoint state $\phi \in H_0^1(\om)\cap L^\infty(\om)$, a multiplier $\bar \zeta $ and $\bar w$ in $\in L^\infty(\om)$ such that the following optimality system is satisfied: 
		\begin{subequations}
			\label{eq:OPT_P'}
			\begin{align}
				& \begin{array}{rll}
					A \bar y &= \bar u +f,  & \text{in } \om, \\
					\bar y & = 0, & \text{on } \ga, 	
				\end{array}\label{eq:state1}\\
				& \begin{array}{rll}
					A^* \bar \phi &= \bar y - y_d,  & \text{in } \om, \\
					\bar \phi & = 0, & \text{on } \ga, 	
				\end{array} \label{eq:adj1}\\
				%
				& (\bar \phi + \beta \, (\delta_\gamma \,\bar \zeta -\bar w) +\alpha \bar u,   u -\bar u  )_{L^2(\om)}  \geq 0, \, \forall u \in U_{ad}	\label{eq:vi}\\
				& \begin{array}{lll}
					\bar\zeta(x) &=1, & \text{ if } \bar u (x) >0, \\
					\bar\zeta(x) &=-1, & \text{ if } \bar u (x) <0, \\
					|\bar\zeta(x)| &\leq 1,  & \text{ if } \bar u (x) =0, \\ 
				\end{array} \label{eq:zeta} 
				\, \text{and}\\ 
				&  \bar w (x): =
				\left\{ \begin{array}{ll}
					\ds \l[ \delta_\gamma - q \l( | \bar u(x)| + \frac{q-1}{\gamma} \r)^{q-1} \r]\sign ( \bar u (x)),& \hbox{if } |\bar u(x)| > \frac{1}{\gamma},\vspace{2mm}\\
					0,&   \text{ otherwise, } 
				\end{array} \right. \label{eq:w}\\
				&\text{ for almost all } x \in \om. \nonumber
			\end{align}
		\end{subequations}
		Moreover, there exist $\lambda_a$ and $\lambda_b$ in $L^2(\om)$ such that the last optimality system can be written as a KKT optimality system:
		\begin{subequations}
			\label{eq:OTP_P'c}
			\begin{align}
				& \begin{array}{rll}
					A \bar y &= \bar u +f  & \text{in } \om, \\
					\bar y & = 0 & \text{on } \ga, 	
				\end{array}\label{eq:state1kkt}\\
				& \begin{array}{rll}
					A^* \bar \phi &= \bar y - y_d  & \text{in } \om, \\
					\bar \phi & = 0 & \text{on } \ga, 	
				\end{array} \label{eq:adj1kkt}\\
				%
				& \bar \phi +\alpha \bar u + \beta \, (\delta_\gamma \, \bar\zeta -\bar w ) + \lambda_b - \lambda_a =0	\label{eq:gradientkkt}\\
				& \begin{array}{ll}
					\lambda_a \geq 0, & \lambda_b \geq 0,  x\\
					\lambda_a (\bar u -u_a) =0, & \lambda_b (u_b-\bar u) =0, \label{eq:compl2}
				\end{array}\\
				& \begin{array}{lll}
					\bar\zeta(x) &=1 & \text{ if } \bar u (x) >0, \\
					\bar\zeta(x) &=-1 & \text{ if } \bar u (x) <0, \\
					|\bar\zeta(x)| &\leq 1  & \text{ if } \bar u (x) =0,
				\end{array}
			\end{align}
		\end{subequations}
		with $\bar w$ given by \eqref{eq:w}.	
	\end{proposition}
	
	The presence of $\bar w$ in this optimality system characterizes the jumps occurring in the solutions. Analyzing this term is crucial for the numerical approximation of \eqref{e:OCP_c}. Note that $w$ represents the superposition operator $u \mapsto w$, defined by $w(x) = j(u(x))$ for almost all $x \in \om$, where $j:\reals \mapsto \reals$, is defined by
	\begin{equation}\label{eq:j}
		j(t)=\l[\delta_\gamma - q\l( |t| + \frac{q-1}{\gamma} \r)^{q-1}\r]\sign(t), \text{ if } |t|>\frac{1}{\gamma},
	\end{equation}
	and  by $0$, otherwise. Also, it satisfies $|j(t)|\leq \delta_{\gamma}$ for all $t$.
	
	\begin{lemma}\label{l:phi_W1p} The \textit{adjoint state} $\bar\phi$, solution of the adjoint equation \eqref{eq:adj1kkt}, satisfies
		\begin{enumerate}[(i)]
			\item $\norm{\bar \phi}_{L^\infty(\om)} +  \norm{\bar \phi}_{H_0^1(\om)} \leq c \norm{ \bar y -y_d}_{L^2(\om)}$, for some positive constant $c$.
			\item There exists $p>2$, such that $\bar \phi\in W_0^{1,p}(\om)$.
		\end{enumerate}
	\end{lemma}
	\begin{proof}
		(i) is a consequence of standard elliptic regularity since $\bar y - y_d \in L^p(\om)$, see \cite{stampacchia65}. On the other hand, notice that by the embedding $H_0^1(\om) \hookrightarrow L^p(\om)$ with $p>2$ implies that $\bar y - y_d \in W^{-1,p}(\om)$. Therefore, owing to \cite[Theorem 2.1.2]{mateos_phd} the adjoint state $\bar \phi$ belongs to $W_0^{1,p}(\om)$ for some $p>2$. Hence, (ii) follows.
	\end{proof}

	\begin{lemma}\label{l:w_prop} Let $j$ be the function defined in \eqref{eq:j}.  The following properties are satisfied:
		\begin{enumerate}[(a)]
			\item $|j(t)| \leq \delta_\gamma$, for all $t \in\reals$.
			\item $j(\cdot)$ is Lipschitz continuous with Lipschitz constant $L_j=2\gamma\delta_\gamma\frac{1}{q}$, i.e.
			$$|j(t_1)-j(t_2)| \leq  L_j |t_1-t_2|, $$ for all $t_1,t_2$ in $\reals$.
		\end{enumerate}
	\end{lemma}
	
	\begin{proof} The first property is a consequence of the definition of $j$ for the case $|t|>\frac{1}{\gamma}$.
		%
		
		Regarding the second property, we analyze the following cases: 
		
		\begin{enumerate}[i)]
			\item  $|t_1|\le\frac{1}{\gamma}$ and $|t_2|\le\frac{1}{\gamma}$,  the result follows without difficulties by the definition of $j$.
			\item \label{22cs} $t_1>\frac{1}{\gamma}$ and $t_2>\frac{1}{\gamma}$ (similarly for $t_1<-\frac{1}{\gamma}$ and $t_2<-\frac{1}{\gamma}$ ), by using the mean value theorem, we get
			\begin{align*}
				|j(t)-j(t_2)|&
				=
				q\left|\left(t_1+\frac{q-1}{\gamma}\right)^{q-1}-\left(t_2+\frac{q-1}{\gamma}\right)^{q-1}\right| \\
				&\le 
				(1-q)q\sup_{r>\frac{1}{\gamma}} \left\{ \left(r+\frac{q-1}{\gamma}\right)^{q-2}\right\}\, |t_1-t_2|
				\\
				%
				&
				\leq
				q^{-1}\gamma\, q^{q}\gamma^{1-q} \, |t_1-t_2|
				\\
				&
				=
				q^{-1}\gamma\delta_{\gamma} \, |t_1-t_2|.
			\end{align*}
			\item \label{33cs}  $|t_1|\le\frac{1}{\gamma}$ and $|t_2|>\frac{1}{\gamma}$ (similarly for $|t_2|\leq\frac{1}{\gamma}$ and $|t_1|>\frac{1}{\gamma}$), we have that $j(t_1)=0$, which implies
			\begin{align*}
				|j(t_1)-j(t_2)|&
				=
				q\left|\left(\frac{q}{\gamma }\right)^{q-1}-\left(t_2+\frac{q-1}{\gamma}\right)^{q-1}\right|,
			\end{align*}
			by applying  the mean value theorem, we get  $ |j(t_1)-j(t_2)| \leq \frac{1}{q} \gamma \delta_\gamma |\frac{1}{\lambda}-t_2|$. Then, using our assumption $0<t_2-1/\lambda\le t_2-t_1$ we obtain	
			\begin{align*}
				|j(t_1)-j(t_2)|
				&
				\le
				\frac{1}{q} \gamma \delta_\gamma, \left|t_2-t_1\right|.
			\end{align*}
			\item \label{44cs} $t_1<-\frac{1}{\gamma}$ and $t_2>\frac{1}{\gamma}$ (similarly for $t_2<-\frac{1}{\gamma}$ and $t_1>\frac{1}{\gamma}$), we have  
			\begin{align*}
				|j(t_1)-j(t_2)|&
				=
				\left|\delta_{\gamma}-q\left(-t_1+\frac{q-1}{\gamma}\right)^{q-1}+\delta_{\gamma}-q\left(t_2+\frac{q-1}{\gamma}\right)^{q-1}\right|
				\\
				&
				\le
				\left|\delta_{\gamma}-q\left(-t_1+\frac{q-1}{\gamma}\right)^{q-1}\right|+\left|\delta_{\gamma}-q\left(t_2+\frac{q-1}{\gamma}\right)^{q-1}\right|
				\\
				&
				=
				q\left|\left(\frac{q}{\gamma }\right)^{q-1}-\left(-t_1+\frac{q-1}{\gamma}\right)^{q-1}\right|
				+ q\left|\left(\frac{q}{\gamma }\right)^{q-1}-\left(t_2+\frac{q-1}{\gamma}\right)^{q-1}\right|,
			\end{align*}
			analogously to the previous cases, the mean value theorem implies that	\begin{align*}
				|j(t_1)-j(t_2)|
				&\le 
				q(1-q)\sup_{|r|>\frac{1}{\gamma}} \left\{ \left(r+\frac{q-1}{\gamma}\right)^{q-2}\right\}\,\left( \left|t_1+\frac{1}{\lambda}\right|+ \left|t_2-\frac{1}{\lambda}\right|\right)
				\\
				&
				\le
				\gamma\delta_{\gamma} \frac{1}{q}\,\left( \left|t_1+\frac{1}{\lambda}\right|+ \left|t_2-\frac{1}{\lambda}\right|\right).
			\end{align*}
			On another hand, since $-1/\lambda<t_2 $ then $-1/\lambda-t_1<t_2-t$. Moreover, since $t_1<1/\lambda$, it follows that $t_2-1/\lambda< t_2-t$. 
			Therefore,
			\begin{align*}
				|j(t_1)-j(t_2)|
				&
				\le
				2\gamma\delta_{\gamma} \frac{1}{q}\, \left|t_2-t_1\right|.
			\end{align*}
		\end{enumerate}
		Thus, we conclude that $j$ is Lipschitz continuous with constant  $2\gamma\delta_{\gamma} \frac{1}{q}$.
	\end{proof}
	\begin{remark}\label{r:auxiliar} For a given function $g$ in $L^2(\om)$, let us consider the following auxiliary $L^1$--sparse optimal control problem: 
		\begin{equation}
			\label{e:OCPL1}
			\begin{cases}
				\displaystyle\min_{(y,u)} J_g(y,u):= ~\frac{1}{2}\| y-y_d \|^2_{L^2(\om)}+\frac{\alpha}{2}\|u\|^2_{L^2(\om)} - \beta\int_{\om} g \, u \, dx + \beta \delta_\gamma \| u \|_{L^1(\om)}\\
				\hbox{ subject to: }\\
				u \in U_{ad} \qquad\text{  and}
				\hspace{20pt}\begin{array}{rll}
					A y=&u + f, &\hbox{in  } \om, \\
					y=&0, &\hbox{on  }  \Gamma.
				\end{array}
			\end{cases}
		\end{equation}
		
		Let $\bar u$ satisfying the optimality system \eqref{eq:OPT_P'} and fix $g=\bar w = j(\bar u)$ in  \eqref{e:OCPL1}. Interestingly, \eqref{eq:OPT_P'} corresponds to the optimality system of the  optimal control problem \eqref{e:OCPL1}, taking $g=\bar w$ defined in \eqref{eq:w}, see \cite{stadler06}.
		Therefore, a local optimal control $\bar u$ for \eqref{e:OCP_c} is also a solution of the problem \eqref{e:OCPL1}. In other words, for a control $\bar u$ satisfying first--order optimality conditions, an auxiliary convex $L^1$ sparse optimal control problem can be associated for which, $\bar u $ is its unique solution if $\alpha>0$, see \cite{stadler06}. 
		
		This observation provides several tools, developed for $L^1$--sparse optimal control problems,  that will serve in the numerical approximation of problem \eqref{e:OCP_c}. Notice, however, the presence of the function $\bar w$ in the optimality system, c.f. \eqref{eq:w}. This function  depends on $\bar u$, which directly affects the overall numerical approximation. The analysis of this element will be crucial for obtaining error estimates in the forthcoming sections.	
	\end{remark}

	Besides the characterization provided by optimality system \eqref{eq:OPT_P'} for an optimal control $\bar u$,  additional valuable information about the optimal control's structure regarding its sparsity properties can be obtained. The lower bound, estimated in the following Proposition, was also proved in \cite[Remark 3 (ii)]{merino2019} for the unconstrained case. Its extension to the constrained case is direct. 

	We apply the maximum principle established in \cite{itoku2008} to problem \eqref{eq:OCP'}. This leads to the following result. 
	\begin{theorem}\label{p:u_softt}
		Let $\bar{u}$ be the optimal control of the problem $(P_\gamma)$ and let $\bar{\phi} \in H_0^1(\om) \cap L^\infty(\om)$ be the associated adjoint state. Then, 
		\begin{enumerate}[(i)]   
			\item   For almost all $x\in\om$,  $\bar u(x)$ satisfies
			\begin{equation*}
				\bar u(x) \in \underset{u \in [u_a,u_b]}{\argmin} \{ (\bar \phi (x) - \beta \bar w(x)) u + \frac{\alpha}{2}u^2 + \beta\delta_\gamma |u| \}.
			\end{equation*}
			%
			\item For almost all $x\in \om$, the value $\bar u(x)$ is characterized by 
			\begin{equation*} 
				\bar u(x) = 
				\begin{cases}
					0, &  \text{if } |\beta\bar w (x)-\bar\phi(x)| \leq \beta\delta_\gamma,\\
					P_{[u_a,u_b]}\l( \frac{\beta \bar w(x) -\bar \phi(x)}{\alpha}- \frac{\beta\delta_\gamma}{\alpha} \sign (\beta \bar w(x) -\bar \phi(x))\r), & \text{otherwise.}
				\end{cases}
			\end{equation*}
			\item  Let $s^*= \Big(\frac{\beta}{\alpha} q(1-q)\Big)^{\frac{1}{2-q}}$, we have  for almost all $x$, $\bar u(x) \not=0$ such that $u_a <\bar u(x)<u_b$, there exists $\gamma_0$ such that the inequality:
			\begin{equation}\label{eq:u_lowup}
				s^*\Big(1+\frac{1}{1-q}\Big) + \frac{1-q}{\gamma}\leq |\bar u(x)| +  \frac{\beta q}{\alpha }  \Big(|\bar u(x)|+\frac{q-1}{\gamma}\Big)^{q-1} \leq \frac{1}{\alpha }\norm{\bar \phi}_{L^\infty(\om)},
			\end{equation}
			holds for all $\gamma>\gamma_0$.
			\item There exist $\rho>0$ and $\gamma_0>0$ such that for all $\gamma>\gamma_0$ we have that 
			the set
			\begin{equation*}
				\Omega_\rho := \{x\in\om : 0<|\bar u(x)| \leq \rho +\frac{1}{\gamma} \}	,
			\end{equation*}
			has zero measure. 
		\end{enumerate}
	\end{theorem}
	
	\begin{proof}
		The proof follows from  \cite[Theorem 2.2]{itoku2014}. Indeed, by choosing $\ell(x,y(x))=\frac{1}{2}(y(x)-y_d)^2$, $h(v)=\frac{\alpha}{2}|v|^2+h_{q,\gamma}^q(v)$, $f(\cdot,y,u(\cdot))=u(\cdot)$, $\langle Ey,\cdot\rangle_{X^*,X}=a(y,\cdot)$ and $X=H_0^1(\om)$, we can verify the hypothesis of \cite[Theorem 2.2]{itoku2014} and conclude that for almost all $x \in \om$: 
		\begin{align*} 
			\bar{u}(x) \in &\underset{u\in [u_a,u_b]}{\argmin} \{\bar{\phi}(x) u+\frac{\alpha}{2} |u|^2+\beta h_{q,\gamma}^q(u)\}\\
			&=\underset{u}{\argmin} \{\bar{\phi}(x) u+\frac{\alpha}{2} |u|^2+\beta h_{q,\gamma}^q(u)+ I_{[u_a,u_b]}\}.
		\end{align*}
		Here, $I_{[u_a,u_b]}$ denotes the indicator function of the interval ${[u_a,u_b]}$. In terms of DC--programming \cite{hiriart1989}, it follows that      
		\begin{align}\label{eq:dc_mp}
			\bar{u}(x) \in \underset{u\in \reals}{\argmin} \{g(u) - h(u)\},
		\end{align}
		where $g(u):= \bar\phi (x) u +\frac{\alpha}{2} u^2 + \beta \delta_\gamma |u| + I_{[u_a,u_b]}$  and $h(u):= \beta h_{q,\gamma}^q(u) - \beta \delta_\gamma\, |u|$ are convex functions. Thus, by considering that  $h'(\bar u(x))=j(\bar u(x))=\bar w(x)$ and noticing that the  optimality condition for the problem \eqref{eq:dc_mp} is  given by $\partial h(\bar u(x)) \subset \partial g(\bar u(x))$, 
		we can find that $\bar u(x)$ fulfills: 
		\begin{equation}
			0 \in \bar \phi (x) - \beta \bar w (x)  + \alpha \bar u(x) + \beta \delta_\gamma \partial |\cdot| (\bar u(x))+\partial I_{[u_a,u_b]} (\bar u(x)), 
		\end{equation}
		which corresponds to the optimality condition of  problem formulated in \textit{(i)}. Characterization \textit{(ii)} follows by considering the \emph{soft--thresholding}  operator (\cite[Example 6.22]{claval2020}) of the convex function $\beta\delta_\gamma |\cdot|$, which is given by 
		\begin{equation}
			\text{prox}_{\beta\delta_\gamma |\cdot|}(y)= \underset{u}{\argmin} \left\{ \frac12 (y-u)^2 + \beta \delta_\gamma |u| \right\} 
			=
			\begin{cases}
				0, & \text{ if } \, |y|\leq \beta \delta_\gamma,\\
				y-\beta \delta_\gamma \sign(y), &  \text{ otherwise. } 
			\end{cases}
		\end{equation}
		Therefore, for the unconstrained case, we may express $\bar u(x)$ with the implicit formula
		\begin{equation}\label{eq:usoft}
			\bar u(x) = 
			\begin{cases}
				0, &  \text{if } |\beta \bar w(x)-\bar\phi(x)|\leq \beta\delta_\gamma,\\
				\frac{\beta \bar w(x) -\bar \phi(x)}{\alpha}- \frac{\beta\delta_\gamma}{\alpha} \sign (\beta \bar w(x) -\bar \phi(x)), & \text{if } |\beta \bar w(x)-\bar\phi(x)|> \beta\delta_\gamma .
			\end{cases}
		\end{equation}
		By considering the lower and upper bounds and complementarity conditions \eqref{eq:compl2}, we arrive to the formula in \textit{(ii)}.
		
		Now, let us deduce relation \eqref{eq:u_lowup}. First, let us consider the set $\om_0=\{x\in \om:0<|\bar u(x)|\leq \frac{1}{\gamma} \}$ and assume that its measure $|\om_0|$ is positive. In this set, by identity \eqref{eq:gradientkkt}, we have that $0=\alpha \bar u(x) + \beta \delta_\gamma\sign(\bar u(x)) -\beta\bar w(x)+ \bar \phi (x)$, implying that $ \norm{\bar \phi}_{L^\infty(\om)}>\beta\delta_\gamma$. 
		On other hand, from elliptic PDEs theory, there exist positive constants $c$ and $C$, such that 
		\begin{align}
			\norm{\bar \phi}_{L^\infty(\om)} \leq c \norm{\bar y - y_d}_{L^2(\om)}  \leq C \norm{y_f - y_d}_{L^2(\om)}, \label{eq:phi_bound}
		\end{align}
		which is a contraction due to the fact that  $\norm{\bar \phi}_{L^\infty(\om)}>\beta\delta_\gamma$ is growing to infinity as $\gamma\to \infty$. 
		

		If $u_b>\bar u(x)>\frac{1}{\gamma}$, we have $\bar u(x)$ minimizes $\frac{\alpha}{2} u^2  + \bar \phi(x) u + \beta (u + \frac{q-1}{\gamma})^q$, therefore $\bar u(x)$ satisfies the critical point condition
		\begin{equation}\label{eq:crit_point}
			{\alpha} \bar u(x)   + \beta q\Big(\bar u(x)  + \frac{q-1}{\gamma}\Big)^{q-1} =-\bar \phi(x).
		\end{equation}
		Since the real function $(\frac{1}{\gamma},+\infty)\ni  s \mapsto \eta(s):=\alpha s + {\beta q}(s + \frac{q-1}{\gamma})^{q-1} $ attains its minimum at $s^* +\frac{1-q}{\gamma}$, where $s^* =   \Big(\frac{\beta}{\alpha} q(1-q)\Big)^{\frac{1}{2-q}} $. Then, we have
		\begin{align*}
			\alpha s^*\Big(1+\frac{1}{1-q}\Big) + \alpha\frac{1-q}{\gamma}=\eta(s^* +&\frac{1-q}{\gamma})=\min_{s>0} \eta(s) \leq\eta (\bar u(x))\\
			&=\alpha \bar u(x) + {\beta q}\Big(\bar u(x) + \frac{q-1}{\gamma}\Big)^{q-1}=-\phi(x). \nonumber
		\end{align*}
		This is analogous for  $u_a<\bar u(x)<0$. Therefore, it follows that 
		\begin{equation}\label{eq:u_lower2}
			s^*\Big(1+\frac{1}{1-q}\Big) + \frac{1-q}{\gamma} \leq |\bar  u(x)| + \frac{\beta q}{\alpha}\Big(|\bar u(x)| + \frac{q-1}{\gamma}\Big)^{q-1}=\frac{1}{\alpha}|\phi(x)|\leq \frac{1}{\alpha}\norm{\bar \phi}_{L^\infty(\om)}. \nonumber
		\end{equation}
		Finally, the part \textit{(iv)} was proved in \cite[Remark 3 (ii)]{merino2019}.
	\end{proof}
	
	\begin{remark}\label{r:u_2_sol}
		The bound given in Theorem \ref{p:u_softt} (iv) is analogous to the bound obtained in \cite{dwachsmuth} for the $L^0$ penalizer. In contrast to the $L^0$ penalizer, we can not use the Hard--thresholding in our analysis, which gives an explicit value for the jump of the optimal control when changing from zero to nonzero values. In our case, the soft--thresholding operator was used to derive part (iii). 
		
		Note that nonzero values of the optimal control necessarily satisfy the critical point condition \eqref{eq:crit_point}, which might be satisfied by two values, see {Figure \ref{fig:1}}. Recall that a unique solution is not guaranteed for problem \eqref{e:OCP_c}. 
		If $\bar u(x)>\frac{1}{\gamma}$, we look for the value that minimizes $\frac{\alpha}{2} u^2  + \bar \phi(x) u + \beta (u + \frac{q-1}{\gamma})^q$. Furthermore, its second--order derivative at $\bar u(x)$ is given by $$ \alpha + \beta q(q-1)\Big(\bar u(x) +  \frac{q-1}{\gamma}\Big)^{q-2}, $$
		which is positive, provided that $\bar u(x)\geq s^*+\frac{1-q}{\gamma}$. By symmetry, the case $\bar u(x)< -\frac{1}{\gamma}$ is analogous.
		Therefore, we choose the solution of \eqref{eq:crit_point}, such that 
		\begin{equation}
			|\bar u(x)|> s^*+ \frac{1-q}{\gamma} \qquad \text{with } s^*= \Big(\frac{\beta}{\alpha} q(1-q)\Big)^{\frac{1}{2-q}},\label{eq:sup_bound}
		\end{equation}
		for almost all $x\in \om$.
		
		\begin{figure}[ht!]
		\includegraphics{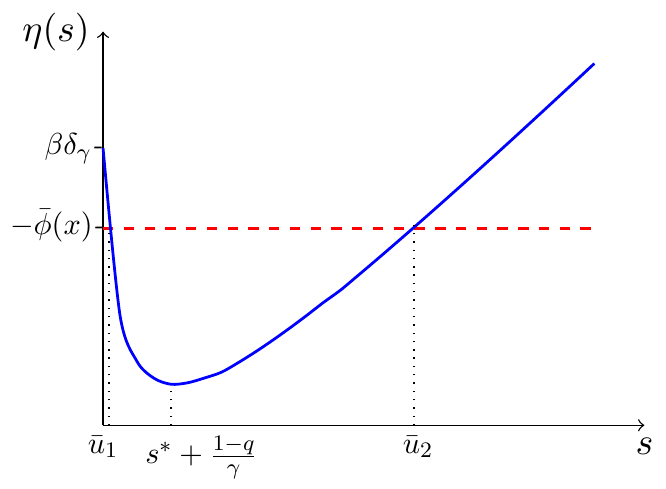}
%
			 \caption{Solutions of the critical point equation \eqref{eq:crit_point}}
			\label{fig:1}
		\end{figure}
	\end{remark}

	In view of Remark \ref{r:u_2_sol} we state the following assumption.
	\begin{assumption}\label{a:local_scc}
		Let $\bar u \in U_{ad}$ satisfying optimality system \eqref{eq:OTP_P'c}.  Then, there exists a constant $C>1$ such that:
		$$|\bar u(x)|\geq  C s^*+ \frac{1-q}{\gamma},$$
		for almost all $x\in\om$, such that $\bar u(x)\not =0$.
	\end{assumption}
	It is clear that (local) optimality for nonsmooth DC--problems can not be established using second-order derivatives due to the lack of differentiability of the objective functional. Instead, we will rely on the underlying convexity given by the DC-representation of the cost functional.  	
	\begin{theorem}{\label{t:ssoc}}
		Let $\bar u \in U_{ad}$ satisfying optimality conditions \eqref{eq:OTP_P'c}. Then, under Assumption \ref{a:local_scc}, there exist constants $\varrho>0$  and $\sigma>0$, such that the following relation holds:
		\begin{equation}
			\sigma \norm{u - \bar u}^2_{L^2(\om)}		\leq J_\gamma(y,u)-J_\gamma(\bar y,\bar u), \quad \forall u \in B_\infty(\varrho,\bar u) \cap U_{ad}, {\label{eq:ssoc}}
		\end{equation}
		for  $\gamma$ sufficiently large.
	\end{theorem}
	\begin{proof} Aiming to prove the local optimality of $\bar u$, we estimate the difference of the cost function at a point $u$ close to $\bar u$. Let $\bar y$ and $y$ be the corresponding states associated with $\bar u$ and $u$, respectively. We consider
		\begin{align}\label{eq:ssc_0}
			\mathcal{J}_\gamma(y,u) - 	\mathcal{J}_\gamma(\bar y,\bar u) ={J}_\gamma(u) - {J}_\gamma(\bar u)= G( u) -H(u)  - G(\bar u) + H(\bar u).
		\end{align}
		Recalling that $G(u)=F(u)+\beta\delta_\gamma \norm{u}_1$, with $F$ the quadratic functional given by $F(u) = \\ \frac12 \norm{Su +Sf-y_d}^2_{L^2(\om)} + \frac{\alpha}{2} \norm{u}^2_{L^2(\om)}$, we have:
		\begin{align*}
			G(u)-G(\bar u) = & (S^*(\bar y -y_d) + \alpha \bar u,u - \bar u)_{L^2(\om)} + (S(\bar u-u), S(\bar u-u) )_{L^2(\om)} \\
			&+  {\alpha} (\bar u -  u,\bar u - u)_{L^2(\om)}  + \beta\delta_\gamma \norm{u}_{L^1(\om)} - \beta\delta_\gamma \norm{\bar u}_{L^1(\om)} 
			\\
			\geq & (\bar \phi + \alpha \bar u + \beta\delta_\gamma \bar \zeta ,u - \bar u)_{L^2(\om)} +  \norm{\bar y - y}^2_{L^2(\om)} + \alpha \norm{\bar u -u}^2_{L^2(\om)},
		\end{align*}
		where the last inequality is obtained from the fact that $\bar \zeta \in \partial \norm{\cdot}_{L^1(\om)}(\bar u)$. 
		Taking into account that $\bar u$ satisfies the variational inequality \eqref{eq:vi}, we deduce
		\begin{equation}\label{eq:ssc_1}
			G(u)-G(\bar u) \geq (\beta \bar w, u-\bar u)_{L^2(\om)} +  \norm{\bar y - y}^2_{L^2(\om)} +   \alpha\norm{\bar u -u}^2_{L^2(\om)}.
		\end{equation}

		On the other hand, by using the G\^ateaux differentiability and the convexity of function $H$, see \cite{merino2019}, we get
		\begin{align}\label{eq:ssc_2}
			H(\bar u) - H(u) \geq H'_G(u, \bar u - u) = (\beta  w, \bar u -u )_{L^2(\om)}, \quad \forall u \in L^2(\om).	
		\end{align}
		Replacing \eqref{eq:ssc_1} and \eqref{eq:ssc_2} in \eqref{eq:ssc_0} we get 
		\begin{align} \label{eq:ssc_3}
			\mathcal{J}_\gamma(y,u) - 	\mathcal{J}_\gamma(\bar y,\bar u) \geq & \norm{\bar y - y}^2_{L^2(\om)} +  \alpha \norm{\bar u -u}^2_{L^2(\om)} - \beta( \bar w - w, \bar u -u)_{L^2(\om)}.  %
		\end{align}
		Notice that the last term in the above inequality is nonnegative due to the monotonicity of $j'$. In view of Assumption \ref{a:local_scc}, we are allowed to chose a $\varrho$, such that for all $u \in B_\infty(\varrho,\bar u)$ they also satisfy 
		\[ |u(x)|\geq Cs^*+\frac{1-q}{\gamma}, \quad \text{for almost all }  x \in \om\backslash\om_\rho. \] 
		Thus, we estimate 
		\begin{align}
			\int_\om &\beta  (\bar w(x)-w(x)) (\bar u(x) -u(x) ) \,dx \nonumber \\
			&\leq \int_{\om\backslash\om_\rho} \beta  |\bar w(x)-w(x)| |\bar u(x) -u(x) | \,dx +\beta \int_{\om_\rho} |\bar w(x)-w(x)||\bar u(x) -u(x)| \,dx \nonumber \\
			& \leq  \beta q (1-q) \int_{\om\backslash\om_\rho}  |\big(\tilde u(x) + \frac{q-1}{\gamma}\big)^{q-2}(\bar u(x)-u(x))||\bar u(x) -u(x)| \,dx +  \beta  \int_{\om_\rho} |w(x)||u(x)| \,dx , \nonumber 
		\end{align}
		where $\tilde u (x)$ lies between $\bar u(x)$ and $u(x)$. Observing that $|\tilde u(x)| \geq Cs^* + \frac{1-q}{\gamma}$ and, using Theorem \ref{p:u_softt}--\emph{(iv)}, we have that $\om_\rho$ has zero measure for $\gamma$ sufficiently large. Therefore, it follows that
		
		\begin{align}\label{eq:ssc_4}
			\int_\om &\beta  (\bar w(x)-w(x)) (\bar u(x) -u(x) ) \,dx <\alpha \norm{\bar u - u}^2_{L^2(\om)}.
		\end{align}
		Replacing \eqref{eq:ssc_4} in \eqref{eq:ssc_3} and taking into account that the bilinear form $(S\cdot,S\cdot)_{L^2(\om)}$ is coercive, there exists $\sigma>0$ satisfying \eqref{eq:ssoc}.
	\end{proof}

	\section{Finite element approximation}
	On $\bar{\om}$, we consider a family of meshes $(\mathcal{T}_h)_{h>0}$ which consist  of triangles $T\in\mathcal{T}_h$ ,  such that the following conditions are satisfied:
	
	\begin{assumption}
		\begin{enumerate}[(i)]
			\item $\displaystyle \bigcup_{T\in\mathcal{T}_h} T=\bar{\om}$ and
			\item  For two triangular elements $T_i$ and $T_j$, $i\not=j$ they share a vertex, a side or are disjoints.
		\end{enumerate}
		For each triangle $T\in\mathcal{T}_h$, we denote  $\rho(T)$ the diameter of  $T$, and $\sigma(T)$ the diameter of the largest ball contained in $T$. The mesh size $h$ associated to the mesh is defined by
		\[
		\displaystyle h=\max_{T\in\mathcal{T}_h}\rho(T).
		\]
	\end{assumption}
	Throughout this paper, we impose the following regularity assumption on the grid: 
	\begin{assumption}\label{h:reggrid}
		There exist two positive constants $\rho$ and $\sigma$ such that 
		\[\frac{\rho(T)}{\sigma(T)}\le \sigma\quad\text{ and }\quad\frac{h}{\rho(T)}\le \rho,\qquad\forall T\in\mathcal{T}_h,\]
		for all $h>0$. 
	\end{assumption}
	Associated with the triangulation $T_h$, we define the following approximation spaces:
	\begin{align}
		Y_h&=\{y_h\in C(\bar{\om}):y_h|_{T}\in P_1(T),\ \forall T\in\mathcal{T}_h,\ y_h=0\ \text{on }\ \Gamma\}, \\
		U_h&=\{u_h\in L^2(\om): u_h|_T=P_0(T), \forall T\in \mathcal{T}_h\},
	\end{align}

	where $P_0(T)$ and $P_1(T)$ denote the set of  real valued constant functions and  linear--affine  continuous real-valued functions defined on $T$, respectively. 
	
	\subsection{Discretization  of the state equation}	 
	The discrete state equation is defined as the following variational problem formulated in $Y_h \subset H_0^1(\om)$: for every $w\in L^2(\om)$, we seek a function $y_h \in Y_h$ satisfying  the equation
	\begin{equation}\label{eq:state_h}
		a(y_h,v_h)=(w,v_h),\quad\forall v_h\in Y_h.
	\end{equation}
	This problem has a unique solution $y_h \in Y_h$ which depends continuously on the data $w$, see \cite{tro2010}. Analogously to the continuous counterpart, we introduce the \emph{discrete} solution operator $S_h: U_h \arrow Y_h$ which assigns to each $w$, the corresponding solution $y_h=y_h(w)$ satisfying \eqref{eq:state_h}. Moreover, the following estimate holds: there exists a positive constant $c$ (independent of $h$ and $w$) such that 
	\begin{equation*}
		\norm{y_h}_{H_0^1(\om)} \leq c \norm{w}_{L^2(\om)}.
	\end{equation*}
	Similarly, we have that the state $y_h$ associated to a control $u\in U_h$ can be written as $y_h = S_h u +y_{h,f}$, with $y_{h,f}=S_hf$. However, for simplicity and without loss of generality, we will assume that $y_{h,f}$ is computed exactly, i.e. $y_{h,f}=y_f$.  
	
	The approximation for linear-elliptic problems of the form \eqref{eq:state} is well known, {see \cite[Section 17]{ciarlet1990}} for the proof of the following rate of approximation.
	\begin{proposition}\label{p:state_estimate}
		Let $y$ and $y_h$ be the solutions, associated to a control $u \in L^2(\om)$, of equations \eqref{eq:state1} and \eqref{eq:state_h}, respectively. There exists a constant $c_A>0$, independent of $h$, such that 
		\begin{equation*}
			\label{esti_l2} ||y-y_h||_{L^2(\om)} + h||y-y_h||_{H^1(\om)}\le c_A h^2. 
		\end{equation*}
		
	\end{proposition}
	
	\subsection{Numerical approximation of the optimal control problem} 
	The approximation of the control functions by piecewise constant functions space $U_h$ is motivated by the discontinuous nature of the solution, as discussed in Theorem \ref{p:u_softt}. Therefore, a discrete control $u \in U_h$ can be written as
	\begin{equation*}
		u=\sum_{T\in\mathcal{T}_h} u_T\chi_T,
	\end{equation*}
	where $\chi_T$ is the characteristic function of $T$ and $u_T \in \reals$, for all $T \in \mathcal{T}_h$. As a result of this type of approximation, the  set of discrete admissible controls is given by:
	\begin{equation}\label{eq:Uadh}
		U^h_{ad} = U_{ad} \cap U_h.
	\end{equation}
	In view of \eqref{eq:state_h} and \eqref{eq:Uadh} we formulate the associated discrete optimal control problem: 
	\begin{equation}
		\tag{$P^h_\gamma$} \label{e:OCP_h}
		\begin{cases}
			\displaystyle\min_{u}~ J_\gamma^h (u):=\frac{1}{2}\| S_h u+S_{h}f-y_d \|^2_{L^2(\om)}+\frac{\alpha}{2}\| u\|^2_{L^2(\om)}+\beta \fsparse_{q,\gamma}(u)\\
			\hbox{ subject to: }\\
			u \in U^h_{ad}.
		\end{cases}
	\end{equation}
	It is clear that for $\alpha >0$ this problem has at least one solution  $\bar u_h$ in $U_h$.
	Similarly to the continuous problem \eqref{e:OCP_c}, we use the indicator function of the discrete admissible set to formulate the discrete reduced problem, as in \eqref{eq:OPT_P'}, which also admits a DC--formulation of the form $J_\gamma^h (u) = G^h(u) - H(u)$. Again, the DC-splitting is given by:
	\begin{align*}
		\begin{array}{ll}
			&\begin{array}{lrlll}
				G^h: & U^h & \arrow &\reals \\
				& u		& \mapsto &  G^h(u) &: = F^h(u) + \beta \delta_\gamma \norm{u}_{L^1(\om)}+{I}_{U^h_{ad}}, \quad\text{and}
			\end{array} \\
			&\begin{array}{lrll}
				H: & U^h & \arrow &\reals \\
				& u		& \mapsto &  H(u) : =  \beta \l( \delta_\gamma \norm{u}_{L^1(\om)} - \fsparse_{q,\gamma}(u) \r),
			\end{array}
		\end{array}
	\end{align*}
	where $F^h(u):= \frac12 \norm{S_h u +S_hf-y_d}^2_{L^2(\om)} + \frac{\alpha}{2} \norm{u}_{L^2(\om)}$. 
	
	Optimality conditions for the discrete problem \ref{e:OCP_h} can also be obtained following the DC--programming; therefore, we omit the proof.
	\begin{proposition}\label{t:fonc_h} Let ${\bar u}_h \in U_h$ be a local solution of \eqref{e:OCP_h}, then there exist:  ${\bar y}_h = S_h \bar u_h + y_{h,f}$ in $Y_h$ (with $y_{h,f}=S_hf$), 
		an adjoint state $\phi_h \in Y_h$, a multiplier $\zeta_h$ and $\bar w_h$ {in $L^\infty(\om)$ } such that the following optimality system is satisfied: 
		\begin{subequations}
			\label{eq:OPT_P'h}
			\begin{align}
				& \begin{array}{rll}
					a(\bar y_h,v) &= (\bar u_h +f,v),  & \forall v\in Y_h,
				\end{array}\label{eq:state1h}\\
				& \begin{array}{rll}
					a^*( \bar \phi_h,v) &= (\bar y_h - y_d,v),  &  \forall v\in Y_h,
				\end{array} \label{eq:adj1h}\\
				%
				& (\bar \phi_h + \beta \, (\delta_\gamma \, \zeta_h -\bar w_h)  +\alpha\,\bar u_h,  u -\bar u_h)_{L^2(\om)} \geq 0, \, \forall u \in U^h_{ad},	\label{eq:vih}\\
				& \begin{array}{lll}
					\zeta_h(x) &=1, & \text{ if } \bar u_h (x) >0, \\
					\zeta_h(x) &=-1, & \text{ if } \bar u_h (x) <0, \\
					|\zeta_h(x)| &\leq 1,  & \text{ if } \bar u_h (x) =0, \\ 
				\end{array} \label{eq:zetah} 
				\, \text{and}\\ 
				&  \bar w_h (x): =
				\left\{ \begin{array}{ll}
					\ds \l[ \delta_\gamma - q \l( | \bar u_h(x)| + \frac{q-1}{\gamma} \r)^{q-1} \r]\sign ( \bar u_h (x)),& \hbox{if } |\bar u_h(x)| > \frac{1}{\gamma},\vspace{2mm}\\
					0,&   \text{ otherwise, } 
				\end{array} \right. \label{eq:wh}\\
				&\text{ for almost all } x \in \om. \nonumber
			\end{align}
		\end{subequations}
		Moreover, there exist ${\lambda_a}_h$ and ${\lambda_b}_h$ in $U_h$ such that the last optimality system can be written as a KKT system by  replacing \eqref{eq:vih} with:
		\begin{subequations}
			\label{eq:OTP_P'c_h}
			\begin{align}
				& \bar \phi_h +\alpha \bar u_h + \beta \, (\delta_\gamma \, \zeta_h -\bar w_h ) + {\lambda_b}_h - {\lambda_a}_h =0	\label{eq:gradientkkt_h}\\
				& \begin{array}{ll}
					{\lambda_a}_h \geq 0, & {\lambda_b}_h \geq 0,  \\
					{\lambda_a}_h(\bar u_h -u_a) =0, & {\lambda_b}_h (u_b-\bar u_h) =0, \label{eq:compl2_h}
				\end{array}\\
				& \begin{array}{lll}
					\zeta_h(x) &=1 & \text{ if } \bar u_h (x) >0, \\
					\zeta_h(x) &=-1 & \text{ if } \bar u_h (x) <0, \\
					|\zeta_h(x)| &\leq 1  & \text{ if } \bar u_h (x) =0,
				\end{array}
			\end{align}
		\end{subequations}
		with $\bar w_h$ given by \eqref{eq:wh}.	
	\end{proposition}
	
	Proceeding as in the continuous case, we have a result on the structure of the discrete optimal control.
	\begin{proposition} \label{p:rho_h}
		There exist $\gamma_0>0$ and $\rho_h>0$ such that for all $\gamma>\gamma_0$ the set
		\begin{equation}\label{eq:Omega_rhohp}
			\Omega_{\rho_h} := \{x\in\om : 0<|\bar u_h(x)| \leq \rho_h +\frac{1}{\gamma}\} 	,
		\end{equation}
		has zero Lebesgue measure. In addition, in view of the maximum principle, the support's bound \eqref{eq:sup_bound} also holds for $\bar u_h$.  
	\end{proposition}
	\begin{proof}
		We follow the arguments from \cite[Remark 3, (ii)]{merino2019} with slight modifications. First observe that $\norm{\bar\phi_h}_{L^\infty(\om)}$ is bounded. Indeed, since the discrete adjoint state satisfies equation \eqref{eq:adj1h}, there exists a constant $c>0$ such that $c\norm{\bar\phi_h}_{L^\infty(\om)} \leq \norm{\bar y_h-y_d}_{L^2(\om)} \leq (2 J_\gamma^h(0))^{\frac12}= \norm{y_{h,f} - y_d}_{L^2(\om)}$.
		
		Now, we claim that the measure of the set $\{x\in\om: 0<|\bar u_h(x)| \leq 1/\gamma \}$ vanishes for sufficiently large $\gamma$. We prove this statement by assuming that the corresponding measure is positive. Without loss of generality, let us also assume that $u_a < \bar u_h (x) < u_b$ for almost all $x$ in $\om_{\rho_h}$. From equation \eqref{eq:gradientkkt_h}, we  infer that for almost all $x$ in  $\om_{\rho_h}$ it holds:
		\begin{equation*}
			|\bar \phi_h(x)| \geq 	\beta \delta_\gamma - \frac{\alpha}{\gamma}.
		\end{equation*}
		This relation contradicts the fact that $|\bar \phi_h(x)|$ essentially bounded whenever $\gamma$ is taken sufficiently large. Thus, our claim is true. 
		
		In the same fashion, we assume that  for all $\rho_h>0$ the set $\{x\in\om: 1/\gamma <|\bar u_h(x)| \leq \rho_h +1/\gamma\}$ has zero Lebesgue measure. Further, we assume that $u_a < \bar u_h (x) < u_b$ and for every $\rho_h>0$ we take $\gamma$ such that $\frac{1}{\gamma} <\rho_h $. Therefore, \eqref{eq:wh} and \eqref{eq:gradientkkt_h} imply
		\begin{align*}
			|\bar \phi_h(x)| = \l| \alpha \bar u_h(x) +\beta q\l( |\bar u_h(x)| + \frac{q-1}{\gamma}\r)^{q-1}\sign(\bar u_h(x))  \r|	
			\geq  \frac{\alpha}{\gamma} + \beta q \l(\rho_h + \frac{q}{\gamma} \r)^{q-1},
		\end{align*}
		which leads us to a contradiction by taking $\rho_h$ sufficiently small. Therefore, $\rho_h$ satisfying property \eqref{eq:Omega_rhohp}  exists.
		
		For the second part of the proposition, we apply the maximum principle, tacitly assuming that the integrals of the cost function are computed exactly. 
		
		By following the lines of the Remark \ref{r:u_2_sol}, we have that for almost all $x\in\om$, where $\bar u_h(x)>0$, the discrete optimal control $\bar u_h(x)$ also satisfies:
		\begin{equation}\label{eq:crit_point2}
			{\alpha} \bar u_h(x)   + \beta q\Big(\bar u_h(x)  + \frac{q-1}{\gamma}\Big)^{q-1} =-\bar \phi_h(x).
		\end{equation}
		Using the same arguments as in the continuous case, we get that $|\bar u_h(x)|\geq s^* +\frac{1-q}{\gamma}$ for almost all $x$ where $|\bar u_h(x)|\not = 0$.
	\end{proof}

	Moreover, we may interpret the solution $\bar u_h$ of \eqref{e:OCP_h} as the solution of the following auxiliar discrete $L^1$--sparse optimal control problem by choosing $\hat g=\bar w_h$. The auxiliar discrete problem reads:
	\begin{equation}
		\label{e:OCPL1_h}
		\begin{cases}
			\displaystyle \min_{u \in U_h}\frac{1}{2}\| S_h u+y_f-y_d \|^2_{L^2(\om)}+\frac{\alpha}{2}\| u\|^2_{L^2(\om)}-\beta (\hat g,u)_{L^2(\om)} + \beta \delta_\gamma \| u \|_{L^1(\om)}\\
			\hbox{ subject to: }\\
			u \in U^h_{ad}.
		\end{cases}
	\end{equation}
	
	Analogously, the former discrete $L^1$--sparse optimal control problem has a unique solution in $U_h$, namely $u_h$ depending of the function $\hat g \in L^2(\om)$.

	Now, we introduce the following interpolation operator to represent functions in $U_h$. 	
	\begin{definition}\label{d:smooth_BV}
		The quasi--interpolant operator $\Pi_h:L^1(\om)\to U_h$ is defined as follows (see \cite{DLReMeVe}):
		\[
		\Pi_h u:=\sum_{T\in\mathcal{T}_h}u_T\chi_T,\quad\text{with}\quad u_T=\frac{1}{|T|}\int_{T} u.
		\]
		The next orthogonality result is also known \cite{wachsmuth}. For $u\in L^2(\om)$, it holds that
		\begin{equation}\nonumber 
			\int_{\om} (u-u_T) \chi_T=0, \quad \forall\, T \in \mathcal{T}_h .
		\end{equation}
		%
		%
		In case of functions belonging to $BV(\om)$, following \cite{bartels15}, we consider the interpolation operator defined on the smooth  approximation $u_\varepsilon$ of $BV$ functions (see \cite[Proposition 2.1]{bartels15}), as follows
		$$ \mathcal{C}_h u : = \Pi_h u_\varepsilon .$$
	\end{definition}
	For convenience in the forthcoming analysis, we fix $\varepsilon =  h$.	
	The interpolation error of $\Pi_h$ is known for $L^2$ and $H^{-1}$ norms  in \cite{DLReMeVe}. Updating the arguments also $L^1$-norm interpolation errors are obtained. See Lemma \ref{l:L1-quasiint} in the Appendix.
	\begin{lemma}\label{est_ele_fin2}
		There exists a positive constant $c$ such that 
		\[
		\|u-\mathcal{C}_h u\|_{L^2(\om)}\leq  c 	h^\frac{1}{2} |Du|(\om)^\frac{1}{2},
		\]
		holds for all $u\in BV(\om)\cap L^2(\om)$.
	\end{lemma}	
	\begin{proof}
		Let $u$ be in $BV(\om)$ and let $c$ be a positive generic constant independent of $h$. We apply Lemma \ref{l:L1-quasiint} with $s=1$ to $u_\varepsilon$. Next, since $C_h u$ is uniformly bounded in $L^\infty$ by its construction \cite{bartels15} and, using the \cite[Proposition 2.1]{bartels15}, we get
		\begin{align}
			\norm{u - \mathcal{C}_h u}_{L^2(\om)} 
			\leq 
			& \norm{u - \mathcal{C}_h u}^{\frac{1}{2}}_{L^\infty(\om)} \norm{u - \mathcal{C}_h u}^{\frac{1}{2}}_{L^1(\om)} \nonumber 
			\\
			\leq
			& c\norm{u}_{L^\infty(\om)}^\frac12 \norm{u - \Pi_h u_\varepsilon}^{\frac{1}{2}}_{L^1(\om)} \nonumber
			\\
			\leq
			& c(\norm{u - u_\varepsilon}_{L^1(\om)}+\norm{u_\varepsilon - \Pi_h u_\varepsilon}_{L^1(\om)})^{\frac{1}{2}} \nonumber
			\\
			\leq &  c(\varepsilon |Du|(\om)+ch\norm{\nabla u_\varepsilon}_{L^1(\om)} )^\frac{1}{2},  \nonumber
			\\
			\leq &  c(\varepsilon +h )^\frac{1}{2}|Du|(\om)^\frac{1}{2}.  \nonumber
		\end{align}
		Recalling that $\varepsilon=h$ in the Definition \ref{d:smooth_BV} of the interpolant $\mathcal{C}_h$ the result follows from assumption \ref{a:cotaDu}.
	\end{proof}
	
	Similarly, we obtain an interpolation error for $\Pi_h$ in the $L^{1}$-norm.
	\begin{corollary}\label{est_ele_fin3}
		There exists a positive constant $c$ such that 
		\[
		\|u-{\Pi}_h u\|_{L^1(\om)}\leq  c h ,
		\]
		holds for all $u\in BV(\om)$.
	\end{corollary}
	
	\begin{proof}
		As in the Lemma \ref{est_ele_fin2}, according to \cite[Proposition 2.1]{bartels15} and Lemma \ref{l:L1-quasiint}, we obtain the following estimation
		\begin{align*}
			\norm{ u -\Pi_h u}_{L^{1}(\om)} 
			&\leq \norm{ u - \mathcal{C}_h u}_{L^{1}(\om)} + \norm{\mathcal{C}_h u- \Pi_h u}_{L^{1}(\om)}
			\\
			&\leq \norm{ u - u_\varepsilon}_{L^{1}(\om)}+ \norm{ u_\varepsilon - \Pi_hu_\varepsilon}_{L^{1}(\om)}+\norm{\Pi_h (u_\varepsilon - u)}_{L^{1}(\om)} 
			\\
			&\leq \varepsilon |Du|(\om) + \norm{u_\varepsilon -\Pi_h u_\varepsilon}_{L^{1}(\om)} + C\norm{ u - u_\varepsilon}_{L^{1}(\om)} 
			\\
			&
			\leq   c(\varepsilon +h )|Du|(\om),  \nonumber
		\end{align*}
		The result follows by considering that $\varepsilon=h$ and Assumption \ref{a:cotaDu}.
	\end{proof}							
	
	The following convergence result will be crucial for our error analysis.	In fact, the second-order optimality conditions established in Theorem \ref{t:ssoc} require the approximation of the local solution in the $L^\infty$ topology.
	
	\begin{proposition}\label{p:phi_estimates} 
		Let $\bar\phi$ and $\bar \phi_h$ be the adjoint and the discrete adjoint states satisfying equations \eqref{eq:state1} and \eqref{eq:adj1h} with controls $\bar u$ and $\bar u_h$, respectively. There exist positive constants $c_A$ and $c$ (independent of $h$) such that the following relations hold:
		\begin{subequations}
			\begin{align}
				\label{adj_esti_h1} &||\bar \phi-\bar \phi_h||_{H_0^1(\om)} \le  c_A (h + \norm{\bar u - \bar u_h }_{L^2(\om)}), \text{ and}
				\\
				\label{adj_esti_inf}& \lim_{h\arrow 0}||\bar \phi-\bar \phi_h||_{L^\infty(\om)} \leq c\lim_{h\arrow 0 } \norm{\bar u- \bar u_h}_{L^2(\om)},
			\end{align}	
		\end{subequations}
		assuming that the last limit exists.
	\end{proposition}
	
	\begin{proof}
		The proof is rather standard since the adjoint equation is linear. Indeed, let $\tilde \phi $ be the solution of the equation 
		$$ a^*(\phi,v)= (S \bar u_h+y_{f} - y_d ,v)_{L^2(\om)}, \quad\forall v\in H_0^1(\om),$$
		by linearity, $\hat \phi:=\bar \phi -\tilde\phi$ satisfies   
		\[
		a^*(\hat \phi,v) = (S\bar u- S \bar u_h,v)_{L^2(\om)}=(S(\bar u- \bar u_h),v)_{L^2(\om)}, \quad\forall v\in H_0^1(\om). 
		\]
		From this equation and standard elliptic estimates, we have that $ \norm{\bar \phi -\tilde\phi}_{H_0^1(\om)} \leq c \norm{\bar u -\bar u_h}_{L^2(\om)}$.  On the other hand, the error estimates for linear elliptic PDE from Proposition \ref{p:state_estimate} also applies to the finite element approximation of the adjoint equation.			
		Hence, there exists a constant $c>0$, such that 
		\begin{align}
			\norm{\bar \phi - \bar \phi_h }_{H_0^1(\om)} \leq &\norm{\bar \phi - \tilde \phi }_{H_0^1(\om)} + \norm{\tilde \phi - \bar \phi_h }_{H_0^1(\om)} \nonumber \\
			= &\norm{\bar \phi - \tilde \phi }_{H_0^1(\om)} + \norm{S^* (S\bar u_h+y_{f} - y_d) - S_h^* (S_h\bar u_h+y_{h,f} - y_d) }_{H_0^1(\om)} \nonumber 
			\\
			\leq &\norm{\bar \phi - \tilde \phi }_{H_0^1(\om)} + \norm{(S^*- S_h^*) (S\bar u_h+y_{f} - y_d)}_{H_0^1(\om)} \nonumber\\
			&\hspace{5cm}
			+ \norm{S_h^* (S\bar u_h+y_{f}-S_h\bar u_h-y_{h,f}) }_{H_0^1(\om)}
			\nonumber \\
			\leq & c(\norm{\bar u -\bar u_h}_{L^2(\om)}+h). \nonumber
		\end{align}
		Next, we prove relation \eqref{adj_esti_inf}. Let us define  $\tilde \phi_h $ as the solution to equation  
		\[ 
		a^*(\phi,v)= (\bar y - y_d,v)_{L^2(\om)}, \quad\forall v\in Y_h.
		\]
		Note that the last equation, defining $\tilde \phi_h$, has the same right-hand of the adjoint equation \eqref{eq:adj1}. We get $\norm{\tilde\phi_h - \bar\phi_h}_{L^\infty(\om)}\leq c \norm{\bar y -\bar y_h}_{L^2(\om)}$. Then, by Lemma \ref{l:phi_W1p} (ii) we know that $\bar \phi \in W_0^{1,p}(\om)$, for some $p>n$ . Thus, using \cite[Theorem 21.5]{ciarlet1990}  we have that $\tilde\phi_h \arrow \bar \phi$ in $L^\infty(\om)$. Therefore
		\begin{align*}
			\lim_{h\arrow 0 } \norm{\bar \phi - \bar\phi_h}_{L^\infty(\om)} \leq & \lim_{h\arrow 0 } (\norm{\bar \phi - \tilde\phi_h}_{L^\infty(\om)}+\norm{\tilde \phi_h - \bar\phi_h}_{L^\infty(\om)}) 
			\\
			\leq &  \lim_{h\arrow 0 } \norm{\bar \phi - \tilde\phi_h}_{L^\infty(\om)}+  c\lim_{h\arrow 0 } \norm{\bar y- \bar y_h}_{L^2(\om)} 
			\\
			\leq &c \lim_{h\arrow 0 } \norm{\bar u- \bar u_h}_{L^2(\om)}.
		\end{align*}
	\end{proof}

	According to Theorem \ref{p:u_softt}, a jump may occur when the optimal control for the continuous problem is not zero. Intuitively, the elements containing this jump should converge to a set of zero measure, provided that the boundary of the function's support is sufficiently regular. Although expected, we do not prove this property, and we will assume that the support's boundary crosses a set of elements with an area proportional to $h$. In this regard, let us introduce the following sets:
	\begin{align}
		&\mathcal{T}_h^0 = \{T \in \mathcal{T}_h: \bar u_T =\Pi_h\bar u(x)|_{T} =  0\},	\text{ and}
		\\
		&\mathcal{T}_h^\# = \{T \in \mathcal{T}_h: |\bar u(x)| >  0, \,\text{a.a.}\, x\in T \}. 
	\end{align}
	
	The following assumption is analogous to \cite[Assumption 3.3]{wachsmuth}, imposed on the structure of the active sets in the context of our discretization.
	\begin{assumption}\label{a:bsupport}
		Let  $\mathcal{T}^*_h : = \mathcal{T}_h \backslash \l( \mathcal{T}_h^0 \cup \mathcal{T}_h^\# \r)$. There exists $c^*>0$ such that 
		\begin{equation}\label{eq:Th_disc}
			|\mathcal{T}^*_h|\leq c^*\,h, \qquad \text{for all } h>0.
		\end{equation}
		
	\end{assumption}

	\begin{theorem}\label{t:convergence}
		Let $\bar{u}_h$ a global solution of problem \eqref{e:OCP_h} defined in a mesh ${\mathcal{T}_h}$ of size $h>0$. Then, the sequence $\{\bar u_h \}$ contains a subsequence $\{\bar u_h \}$(without renaming) in $U_h$, such that
		\begin{enumerate}[(i)]
			\item    $ \lim_{h\to 0}J_h(\bar{u}_h)=\inf \{ J({u}): u \in U_{ad}\cap BV(\om)\},$ 
			\item    $\bar u_h$ converges  to a solution $\bar u$ of \eqref{e:OCP_c} in $L^2(\om)$.
			\item    Under Assumptions \ref{a:local_scc} and \ref{a:bsupport}, $\lim_{h\to 0} \norm{\bar u_h -\bar u}_{L^\infty(\om)}=0$. 
		\end{enumerate}
	\end{theorem}
	\begin{proof}
		(i) 	Let $\{\bar u_h\}_{h} \subset U_{ad}^h \subset BV(\om)$	be a sequence of solutions \eqref{e:OCP_h} (with discretization parameter $h>0$) and associated sequence of corresponding states $\{\bar y_h\}_h$ in $H_0^1(\om)$. By the definition of $U_{ad}^h$ and the Tikhonov term in the cost function, it follows that $\{\bar u_h\}_{h}$  is a bounded sequence in $BV(\om)\cap L^2(\om)$. Also, we confirm that $\{\bar y_h\}_h$ is a bounded sequence in $H_0^1(\om)$ given our assumption and the inequality:
		\begin{align}
			\norm{\bar y_h}_{H_0^1(\om)} =\norm{S_h\bar u_h}_{H_0^1(\om)} \leq &\norm{(S_h- S)\bar u_h}_{H_0^1(\om)} + \norm{S\bar u_h}_{H_0^1(\om)} \nonumber\\
			\leq &c_A h + c\norm{\bar u_h}_{L^2(\om)}.
		\end{align}
		Therefore, we subtract (without renaming)  weakly* convergent subsequence  $\{\bar u_h\}_{h}$ in $BV(\om)$, such that $\bar u_h \overset{\ast}{\rightharpoonup} \tilde u $ in $BV(\om)$ and $\bar u_h {\rightharpoonup} \tilde u $ in $L^2(\om)$. Here $\tilde u \in U_{ad}$ because the admissible set $U_{ad}$ is weakly closed in $L^2(\om)$.  The sequence of corresponding states: $\{\bar y_h\}_h$ is such that $\bar y_h \rightharpoonup \tilde y$ in $H_0^1(\om)$, as $h \arrow 0$. Moreover, the compact embedding of $H_0^1(\om)$ into $L^2(\om)$ implies that $\bar y_h \arrow \tilde y$ in $L^2(\om)$ as $h\arrow 0$.
		
		Using the fact that $\bar u_h \overset{\ast}{\rightharpoonup} \tilde u $ in $BV(\om)$ implies that $\bar u_h \arrow \bar u$ in $L^1(\om)$ and by  \cite[Lemma 2]{merino2019} we have also that $\fsparse_{q,\gamma}(\bar u_h) \arrow \fsparse_{q,\gamma}(\tilde u)$. Then, by using the convexity of $F^h$ and the fact that $\bar u_h$ is a solution for \eqref{e:OCP_h}, we estimate:
		\begin{align}
			J_\gamma(\tilde u)=\frac{1}{2}&\| \tilde y+y_f-y_d \|^2_{L^2(\om)}+\frac{\alpha}{2}\| \tilde u\|^2_{L^2(\om)}+\beta \fsparse_{q,\gamma}(\tilde u) \nonumber \\
			\leq & \liminf_{h\arrow 0} \l\{ \frac{1}{2}\| \bar y_h+y_f-y_d \|^2_{L^2(\om)}+\frac{\alpha}{2}\| \bar u_h\|^2_{L^2(\om)} \r\} + \beta \fsparse_{q,\gamma}(\tilde u)    \nonumber \\
			\leq & \liminf_{h\arrow 0} \l\{ F^h(\bar u_h) +  \beta \fsparse_{q,\gamma}(\bar u_h)  \r\} +\beta \fsparse_{q,\gamma}(\tilde u) - \liminf_{h\arrow 0} \beta \fsparse_{q,\gamma}(\bar u_h) \nonumber\\
			%
			%
			= & \liminf_{h\arrow 0} J^h_\gamma (\bar u_h) \nonumber \\
			\leq & \liminf_{h\arrow 0} J^h_\gamma (\mathcal{C}_h  \bar u), \label{eq:convergence}	%
		\end{align}
		Moreover, by continuity of $S$ we have that $S\tilde u=\tilde y$, since $S\bar u_h \arrow S\tilde u$   and $\bar u_h \arrow \tilde u$ in $L^1(\om)$, respectively. Taking into account that $\mathcal{C}_h u\arrow u$ in $L^2(\om)$ and $S\mathcal{C}_h \bar u \arrow S\bar u$ in $L^2(\om)$ as $h\arrow 0$, it follows that $F^h(\mathcal{C}_h \bar u) \arrow  F(\bar  u)$. Hence, from \eqref{eq:convergence}, we obtain
		\[
		J_\gamma(\tilde u)	\leq J_\gamma( \bar u), \quad\forall u \in U_{ad}\cap BV(\om). \nonumber
		\]
		Since $\tilde u \in U_{ad}\cap BV(\om)$ then $\tilde  u$ is an optimal control for $\eqref{e:OCP_c}$. Therefore, we use the notation $\bar u =\tilde u$ and $\bar y =\tilde y$. 
		
		(ii) Using Taylor's expansion of the quadratic form $F(u)=~\frac{1}{2}\| y-y_d \|^2_{L^2(\om)}+\frac{\alpha}{2}\|u\|^2_{L^2(\om)}$ we have the relation
		\begin{eqnarray}
			&F(\bar u_h) -F(\bar u)&=F'(\bar u)(\bar u_h -\bar u) + \frac{1}{2}\| \bar y_h-\bar y \|^2_{L^2(\om)}+\frac{\alpha}{2}\|\bar u-\bar u_h\|^2_{L^2(\om)} \nonumber
			\\
			&	&=(\alpha \bar u + \bar\phi, \bar u_h -\bar u )_{L^2(\om)}  + \frac{1}{2}\| \bar y_h-\bar y \|^2_{L^2(\om)}+\frac{\alpha}{2}\|\bar u-\bar u_h\|^2_{L^2(\om)},\nonumber
		\end{eqnarray}		
		which is used to deduce:
		\begin{eqnarray}
			&\frac{\alpha}{2}\|\bar u_h -\bar u\|^2_{L^2(\om)} &= 
			F(\bar u_h)- F(\bar u) - (\alpha \bar u + \bar \phi, \bar u_h -\bar u )_{L^2(\om)}  - \frac{1}{2}\| \bar y_h-\bar y \|^2_{L^2(\om)} \nonumber \\
			&&= J_\gamma(\bar u_h)- J_\gamma(\bar u)  - (\alpha \bar u + \bar \phi, \bar u_h -\bar u )_{L^2(\om)}   - \frac{1}{2}\| \bar y_h-\bar y \|^2_{L^2(\om)} \nonumber \\
			&&\qquad\qquad  + \beta \fsparse_{q,\gamma} (\bar u) - \beta \fsparse_{q,\gamma} (\bar u_h).\nonumber 
		\end{eqnarray}
		Recalling that $\bar u_h \rightharpoonup \bar u$ in $L^2(\om)$ and  $\bar u_h \arrow \bar u$ in $L^1(\om)$, we have by (i) that $J_\gamma(\bar u_h) \arrow J_\gamma(\bar u)$ and  $(\alpha \bar u + \bar \phi, \bar u_h -\bar u )_{L^2(\om)} \arrow 0$. Furthermore, we have that $\beta \fsparse (\bar u) \arrow \beta \fsparse (\bar u_h)$ by  \cite[Lemma 2]{merino2019} and  $\bar y_h \arrow \bar y$ in $L^2(\om)$. Passing to the limit altogether gives the result.

		In order to proof (iii) we use Assumption \ref{a:local_scc}, let us first consider $T \in \mathcal{T}_h^0 = \{T \in \mathcal{T}_h: \bar u_T =\Pi_h\bar u(x)|_{T} =  0\}$, therefore 
		\begin{align}
			|\bar u_{h_T}|=\frac{1}{|T|}\l|\int_T \bar u_{h}(x) dx\r| =& \frac{1}{|T|}\l|\int_T \bar u_{h}(x) -\bar u_T dx \r|\nonumber \\
			=& \frac{1}{|T|}\l|\int_T \bar u_{h}(x) - \Pi_h\bar u (x) dx \r| \nonumber \\
			\leq &\frac{1}{|T|}\l[\l|\int_T \bar u_{h}(x) - \bar u(x) dx \r| + \l|\int_T \bar u(x) - \Pi_h \bar u(x) dx \r|\r] \nonumber \\
			\leq &\frac{1}{|T|}\l[ \int_T |\bar u_{h}(x) - \bar u(x)| dx +  |T|\norm{\bar u - \Pi_h \bar u}_{L^2(\om)}\r],\nonumber
		\end{align}
		by using Lemma \ref{est_ele_fin2}, we get
		\begin{align*}
			|\bar u_{h_T}|\leq &
			\frac{1}{|T|}\int_T |\bar u_{h}(x) - \bar u(x) |dx  +  \norm{\bar u - \mathcal{C}_h \bar u}_{L^2(\om)} +   \norm{ \Pi_h (\bar u_\varepsilon -  \bar u)}_{L^2(\om)} \\ 
			\leq &\frac{1}{|T|}\int_T |\bar u_{h}(x) - \bar u(x) |dx  +  c 	h^\frac{1}{2} |Du|(\om)^\frac{1}{2} +   ch.
		\end{align*}
		From this relation and since $\bar u_h \arrow \bar u$ in $L^1(\om)$ as $h\arrow 0$ it follows  that $|\bar u_{h_T}| < \frac{1}{\gamma}$ for $h$ sufficiently small, we have that 	$\bar u_{h_T}=0$.	
		
		Now, let us analyze the case  $T \in \mathcal{T}_h^\#$, since  $\bar u_h \arrow \bar u$ in $L^1(\om)$, there exists  $h_1>0$ such that $|\bar u_{h_T}-\bar u_T|=|\bar u_{h_T}-\frac{1}{|T|}\int_T\bar u (x) dx|<\epsilon$, for all $h<h_1$. Moreover, by noticing that $|\bar u_T| = |\frac{1}{|T|}\int_T\bar u (x) dx |>\rho + \frac{1}{\gamma}$,  we have that $\sign(\bar u_{h_T}) = \sign(\bar u_T)\not =0$.  Consequently, $\sign(\bar w(x)-\bar \phi(x))=\sign(\bar w_h(x)-\bar \phi_h(x))$. Therefore, using Theorem  \ref{p:u_softt} it follows that
		\begin{align}
			|\bar u_{h}(x)-\bar u (x)| =  &\Big |P_{[u_a,u_b]}\Big(\frac{\beta \bar w(x)-\bar \phi(x)}{\alpha} -\frac{\beta \delta_\gamma}{\alpha} \sign(\bar w(x)-\bar \phi(x))\Big) \nonumber \\ 
			&-P_{[u_a,u_b]}\Big(\frac{\beta \bar w_h(x)-\bar \phi_h(x)}{\alpha} -\frac{\beta \delta_\gamma}{\alpha} \sign(\bar w_h(x)-\bar \phi_h(x)) \Big) \Big| \nonumber \\
			\leq & \l|\frac{\beta \bar w(x)-\bar \phi(x)}{\alpha} -  \frac{\beta \bar w_h(x)-\bar \phi_h(x)}{\alpha}\r| \nonumber \\
			\leq & \frac{\beta}{\alpha} |\bar w(x) -\bar w_h(x)| + \frac{1}{\alpha}|\bar \phi(x)-\bar \phi_h(x)|. \label{eq:linf_conv.1}
		\end{align}
		By the definition of $\bar w_h$,  we have that $\bar w_h= j(\bar u_h) = j(\bar u_{h_T})$, for almost all $x\in T$. Then, denoting $u_\tau= \bar u(x)+ \tau(\bar u_{h_T}-\bar u(x))$, with $\tau\in (0,1)$, and considering that $\sign(\bar u_{h_T}) = \sign(\bar u_T)\not =0$ the following estimate holds:
		\begin{align}
			|\bar w(x) -\bar w_h(x)|  \leq &  	 q\left|\Big(|\bar u(x)|+ \frac{q-1}{\gamma}\Big)^{q-1} - \Big(|\bar u_{h_T}|+ \frac{q-1}{\gamma}\Big)^{q-1} \right|\nonumber \\
			= &q (1-q)|\bar u(x) -\bar u_{h_T}| \int_0^1 \Big(|u_\tau| + \frac{q-1}{\gamma}\Big)^{q-2} \,d\tau\, \nonumber  \\
			=& c_T(x) |\bar u(x) -\bar u_{h_T}|,  \nonumber 
		\end{align}
		where $c_T: T \arrow \reals$ is given by $c_T(x):=q (1-q)\ds\int_0^1 \Big(|u_\tau (x)| + \frac{q-1}{\gamma}\Big)^{q-2} \,d\tau$.  In addition, by Assumption \ref{a:local_scc} and  construction of $u_\tau$, there exist $C_0>1$ such that $|u_\tau (x)|>c_0 s^* +\frac{1-q}\gamma$. Then,  there exists $c_0<1$ such that  $\norm{c_T}_{L^\infty(T)}<c_0 \frac{\alpha}{\beta}$  for all $T\in\mathcal{T}^\#$. Therefore  
		\begin{align}
			|\bar w(x) -\bar w_h(x)|  <c_0\frac{\alpha}{\beta} |\bar u(x) -\bar u_{h_T}| =c_0 \frac{\alpha}{\beta} |\bar u(x) -\bar u_{h} (x)|. \nonumber 
		\end{align}
		Since $x$ is fixed, the last relation can be replaced in \eqref{eq:linf_conv.1}, implying that 
		\begin{align*}
			|\bar u_{h}(x)-\bar u (x)|\leq c_1|\bar \phi(x)-\bar \phi_h(x)|,
		\end{align*} 
		for some constant $c_1>0$ independent of $h$. Moreover, we can estimate 
		\begin{align}
			|\bar u_{h}(x)-\bar u (x)|\leq & c_1|\bar \phi(x)-\bar \phi_h(x)| \nonumber\\ 
			\leq &  c_1\|\bar \phi-\bar \phi_h\|_{L^\infty(T)} \nonumber \\
			\leq &  c_1\|\bar \phi-\bar \phi_h\|_{L^\infty(\om)}. \nonumber
		\end{align} 
		Finally, in view of Proposition \ref{p:phi_estimates} and (ii), taking the limit $h\arrow \infty$  we obtain
		\begin{align}
			\lim_{h\arrow 0}|\bar u_{h}(x)-\bar u (x)|
			\leq &  c_1\lim_{h\arrow 0}\|\bar \phi-\bar \phi_h\|_{L^\infty(\om)} \leq c \lim_{h\arrow 0} \norm{\bar u -\bar u_h}_{L^2(\om)} =0 . \nonumber
		\end{align} \qedhere
	\end{proof}
	\section{Error estimates}
	
	In what follows, let us denote by $\{\bar u_h \}$, the sequence of global solutions of problems \eqref{e:OCP_h} such that, under Assumptions \ref{a:local_scc} and \ref{a:bsupport}, $\lim_{h\to 0} \norm{\bar u_h -\bar u}_{L^\infty(\om)}=0$, where $\bar u$  is a solution of \eqref{e:OCP_c}, see Theorem \ref{p:u_softt}. We will obtain the estimate of $\bar u-\bar u_h$ in the $L^2$ norm using optimality conditions. We start this section with the following lemma, which invokes necessary optimality conditions \eqref{eq:vi} and \eqref{eq:vih}.

	\begin{lemma}\label{l:aux_stable}
		Let $g $ and $\hat g$ be functions in  $L^2(\om)$, and let $u$ and ${\hat{u}_h}$ be optimal controls, with corresponding states $y$ and $\hat{y}_h$, the solutions of \eqref{e:OCPL1} and \eqref{e:OCPL1_h} associated to $g$ and $\hat g$, respectively. Then, for all $v \in U_{ad}$ and all $v_h\in U^h_{ad}$ the following estimate holds:
		\begin{align*}
			\alpha\norm{u - {\hat u_h}}&_{L^2(\om)}^2 +\|{y}-\hat{y}_h\|_{L^2}^2 \nonumber \\
			&\leq ( \phi , v -{\hat{u}_h} + v_h - u)_{L^2(\om)}  +\alpha\,(  u, v -{\hat{u}_h} + v_h - u)_{L^2(\om)} \nonumber \\
			& \quad +\alpha (\hat{u}_h-u, v_h - u) )_{L^2(\om)}  \nonumber \\
			&\quad +(\hat{y}_h-{y},(S_h-S)v_h+S(v_h-u))_{L^2(\om)}+(y-y_d,(S_h-S)(v_h-\hat{u}_h)_{L^2(\om)} \nonumber \\
			& \quad + \beta\delta_\gamma (\norm{v}_{L^1(\om)}- \norm{u}_{L^1(\om)} + \norm{v_h}_{L^1(\om)}- \norm{\hat{u}_h}_{L^1(\om)} ) \nonumber\\
			& \quad -\beta(g,v -u)_{L^2(\om)}-\beta (\hat g,v_h -\hat{u}_h)_{L^2(\om)}. 
		\end{align*}
	\end{lemma}
	\begin{proof} Let $u$ and $\hat{u}_h$ satisfying their corresponding first optimality conditions, with states $y = Su +y_f$ and $\hat y_h=S_h\hat{u}_h+y_f$, adjoint states $\phi \in H_0^1(\om)$ and $\hat \phi_h \in Y_h$,  multipliers $\zeta$ and $\hat\zeta$, respectively. Then, we have
		\begin{align}
			( \phi  + \beta \, (\delta_\gamma \, \zeta - g ), v -u )_{L^2(\om)} &+ \alpha\,(  u,  v -u )_{L^2(\om)}  \geq 0,&&   \forall v \in U_{ad} \quad\text{and} \nonumber\\
			( \hat\phi_h + \beta \, (\delta_\gamma \, \hat\zeta -  \hat g ), v_h -\hat{u}_h )_{L^2(\om)}  &+ \alpha\,(  \hat u_h,  v_h -\hat u_h )_{L^2(\om)}\geq 0,&&  \forall v_h \in U^h_{ad}, \nonumber
		\end{align}
		by adding these both inequalities we get that
		\begin{align}
			0 \leq &	( \phi, v-\hat{u}_h )_{L^2(\om)} + \alpha( u, v-\hat{u}_h )_{L^2(\om)} + ( \phi,  \hat{u}_h - u)_{L^2(\om)} + \alpha( u, \hat{u}_h - u))_{L^2(\om)} \nonumber \\
			&+( \hat\phi_h, v_h - u )_{L^2(\om)}+  \alpha( \hat u_h , v_h - u)_{L^2(\om)}  + ( \hat\phi_h,   u-\hat{u}_h)_{L^2(\om)} + \alpha(\hat{u}_h, u-\hat{u}_h)_{L^2(\om)}  \nonumber \\
			& + \beta (\delta_\gamma \, \zeta - g , v-u)_{L^2(\om)} + \beta (\delta_\gamma \, \hat\zeta - \hat g , v_h -\hat{u}_h)_{L^2(\om)}. \nonumber
		\end{align}
		Taking the right terms conveniently to the left--hand side, it follows that
		\begin{align}
			\alpha  (u-\hat{u}_h,u-&\hat{u}_h)_{L^2(\om)} \nonumber \leq ( \phi, v-\hat{u}_h )_{L^2(\om)} + \alpha(u, v-\hat{u}_h )_{L^2(\om)} \nonumber \\
			&\quad+ ( \hat\phi_h, v_h - u )_{L^2(\om)}+  \alpha( \hat u_h , v_h - u)_{L^2(\om)} + (\phi -\hat\phi_h, \hat{u}_h - u)_{L^2(\om)} \nonumber \\
			& \quad + \beta (\delta_\gamma \, \zeta - g , v-u)_{L^2(\om)} + \beta (\delta_\gamma \, \hat\zeta - \hat g, v_h -\hat{u}_h)_{L^2(\om)}. \label{eq:estimatel11}
		\end{align}
		Let us focus on the last two terms. Since $\zeta \in \partial \norm{\cdot}_{L^1(\om)}(u)$ and $\hat \zeta \in \partial \norm{\cdot}_{L^1(\om)}(\hat{u}_h)$, we obtain:
		\begin{align}
			\beta (\delta_\gamma \, & \zeta - g , v-u)_{L^2(\om)} + \beta (\delta_\gamma \, \hat\zeta - \hat g, v_h -\hat{u}_h)_{L^2(\om)} = \beta\delta_\gamma (\zeta , v-u)_{L^2(\om)} \nonumber\\
			& \qquad\qquad\qquad -\beta(g,v-u)_{L^2(\om)}  + \beta\delta_\gamma (\hat\zeta , v_h -\hat{u}_h)_{L^2(\om)} -\beta (\hat g,v -\hat{u}_h)_{L^2(\om)}\nonumber \\
			& \leq \beta\delta_\gamma (\norm{v}_{L^1(\om)}- \norm{u}_{L^1(\om)} + \norm{v_h}_{L^1(\om)}- \norm{\hat{u}_h}_{L^1(\om)} ) \nonumber \\
			& \qquad \qquad\qquad-\beta(g,v-u)_{L^2(\om)}-\beta (\hat g,v_h -\hat{u}_h)_{L^2(\om)}. \label{eq:estimatel12}
		\end{align}
		Now, by considering the adjoint--state terms in \eqref{eq:estimatel11}, it follows that
		\begin{align}
			( \phi, v-\hat{u}_h )&_{L^2(\om)}  + (\hat\phi_h, v_h - u )_{L^2(\om)} + (\phi -\hat\phi_h, \hat{u}_h - u)_{L^2(\om)} \nonumber \\
			=&( \phi, v-\hat{u}_h + v_h - u)_{L^2(\om)}  + (\phi-\hat\phi_h, u -v_h )_{L^2(\om)}  + (\phi -\hat\phi_h, \hat{u}_h - u)_{L^2(\om)} \nonumber \\
			=& ( \phi, v-\hat{u}_h + v_h - u)_{L^2(\om)} -  (\phi, v_h-\hat{u}_h)_{L^2(\om)} + (\hat\phi_h, v_h-\hat{u}_h)_{L^2(\om)}. \nonumber
		\end{align}
		Here, we use the fact that $\phi = S^*(y_g -y_d)$ and $\hat\phi_h = S_h^*(\hat y_h -y_d)$ to get 
		\begin{align}
			( \phi , v-&\hat{u}_h + v_h - u)_{L^2(\om)} \nonumber\\
			&-  (y-y_d, Sv_h-S\hat{u}_h)_{L^2(\om)} + (\hat y_h-{y_d}, S_hv_h-S_h\hat{u}_h)_{L^2(\om)}\nonumber\\
			=& ( \phi, v-\hat{u}_h + v_h - u)_{L^2(\om)} -  (y-y_d, Sv_h-S\hat{u}_h)_{L^2(\om)} + (\hat y_h-{y_d}, S_hv_h-\hat y_h)_{L^2(\om)} \nonumber\\
			=& ( \phi, v-\hat{u}_h + v_h - u)_{L^2(\om)} + (y-y_d,(S_h-S)(v_h-\hat{u}_h))_{L^2(\om)}   \nonumber\\
			& - ( \hat y_h-y, y - \hat y_h )_{L^2(\om)} + (\hat{y}_h-{y},(S_h-S)v_h+S(v_h-u))_{L^2(\om)}.
			\label{eq:estimatel13}
		\end{align}
		Furthermore,  we have that
		\begin{align}
			\alpha( u, v-\hat{u}_h )_{L^2(\om)} +  \alpha( \hat u_h , v_h - u)_{L^2(\om)}  = \alpha( u, v-&\hat{u}_h + v_h - u )_{L^2(\om)}  \nonumber \\
			&+ \alpha(\hat u_h-u , v_h - u)_{L^2(\om)} 
			\label{eq:estimatel14}
		\end{align}
		Replacing relations  \eqref{eq:estimatel13}, \eqref{eq:estimatel12} and \eqref{eq:estimatel14} in \eqref{eq:estimatel11} we conclude the result.
	\end{proof}
	\begin{lemma}\label{l:aux_stable2} Let $g $ and $\hat g$ be functions in  $L^\infty(\om)$, and let $u$ and ${\hat{u}_h}$ , with corresponding states $y$ and $\hat{y}_h$, the solutions of \eqref{e:OCPL1} and \eqref{e:OCPL1_h} associated to $g$ and $\hat g$, respectively. Then, there exists a constant $\hat c(\epsilon)$ independent of $h$, such that
		\begin{align*}
			(1-\epsilon){\alpha}\norm{u - {\hat u_h} }_{L^2(\om)}^2 +\frac{1}{2}\|{y}-\hat{y}_h\|_{L^2}^2 \leq &  \,\hat c(\epsilon) h + \beta (g-\hat g, u -\hat{u}_h)_{L^2(\om)},
		\end{align*}
		for some $\epsilon \in (0,1)$.
	\end{lemma}
	
	\begin{proof}
		Following the argument from \cite[Lemma 4.2]{wachsmuth}, it follows that
		$$
		\|\Pi_h{u}\|_{L^1}
		=\|\sum_{T\in \mathcal{T}_h} {u}_T \chi_T\|_{L^1}
		\le\sum_{T\in \mathcal{T}_h} \int_{\Omega} |{u}_T|\ \chi_T
		=\sum_{T\in \mathcal{T}_h} \frac{\ds\int_{T} {|u|} \,dx }{|T|} |T|
		%
		%
		=\|{u}\|_{L^1}.
		$$
		In addition, by the inclusion $U_{ad}^h \subset U_{ad}$, let us consider Lemma \ref{l:aux_stable} with $v = \hat{u}_h$ and $v_h = \Pi_h u$. Thus, we get the relation
		\begin{align}
			\alpha\norm{ u - {\hat u_h} }&_{L^2(\om)}^2 +\|{y}-\hat{y}_h\|_{L^2}^2 \nonumber \\
			&\leq ( \phi, \Pi_h u - u)_{L^2(\om)}+\alpha\, ( u, \Pi_h u - u)_{L^2(\om)}+ \alpha (\hat{u}_h-u, \Pi_h u - u )_{L^2(\om)}  \nonumber 
			\\
			&\quad +(\hat{y}_h-{y},(S_h-S)\Pi_h u+S(\Pi_h u-u))_{L^2(\om)} \nonumber 
			\\
			& \quad +(y-y_d,(S_h-S)(\Pi_h{u}_h-\hat{u}_h))_{L^2(\om)}+ \beta\delta_\gamma (\norm{\Pi_h u}_{L^1(\om)}- \norm{u}_{L^1(\om)}) \nonumber
			\\ 
			&\quad-\beta(g,\hat{u}_h -u)_{L^2(\om)}-\beta (\hat g, \Pi_hu -\hat{u}_h)_{L^2(\om)} \nonumber
			\\
			%
			&\leq \norm{\phi}_{L^\infty(\om)}\norm{\Pi_h u - u}_{L^{1}(\om)} +\alpha\norm{u}_{L^\infty(\om)}\norm{\Pi_h u - u}_{L^{1}(\om)} \nonumber 
			\\
			&\quad + \alpha (\hat{u}_h-u, \Pi_h u - u )_{L^2(\om)}  +(\hat{y}_h-{y},(S_h-S)\Pi_h u+S(\Pi_h u-u))_{L^2(\om)} \nonumber\\
			& \quad +\norm{y-y_d}_{L^2(\om)} \norm{(S_h-S)(\Pi_h u-\hat{u}_h)}_{L^2(\om)}\nonumber \\ 
			&\quad -\beta(g,\hat{u}_h -u)_{L^2(\om)}-\beta (\hat g, \Pi_hu -\hat{u}_h)_{L^2(\om)}. \label{eq:lem2_0}
		\end{align}
		Notice that the $L^1$--norm terms have vanished. Moreover by using the Corollary  \ref{est_ele_fin3} for $\norm{\Pi_h u - u}_{L^{1}(\om)}$, we are able to get a estimation for the first and second  term given by $h$.
		As for third term, due to Young's inequality  and the Corollary \ref{est_ele_fin3} we have that
		\begin{align}
			\alpha(\hat{u}_h-u, \Pi_h u - u)_{L^2(\om)} \leq & \epsilon{\alpha} \norm{ \hat{u}_h-u}^2_{L^2(\om)} + \frac{\alpha}{4\epsilon}\norm{\Pi_h u - u}^2_{L^2(\om)} \nonumber
			\\
			\leq & \epsilon {\alpha} \norm{\hat{u}_h-u}^2_{L^2(\om)} + c \frac{\alpha}{4\epsilon} \norm{\Pi_h u - u}_{L^1(\om)} \nonumber
			\\\leq & \epsilon {\alpha} \norm{\hat{u}_h-u}^2_{L^2(\om)} + c \frac{\alpha}{4\epsilon}  h \label{eq:lem2_1},
		\end{align}
		for any $\epsilon \in (0,1)$. 
		Again, by using Young's inequality, it follows that 
		\begin{align*}
			(\hat{y}_h-{y},(S_h-S)\Pi_h u +S\hat{u}_h-{u}))&\leq \frac{1}{2}\|{y}-\hat{y}_h\|^2_{L^2(\om)}
			+\frac{1}{2}\|(S_h-S)\Pi_h u+S(\Pi_h u-u)\|^2_{L^2(\om)}
			\\
			&=\frac{1}{2}\|{y}-\hat{y}_h\|^2_{L^2(\om)}
			+\|(S_h-S)\Pi_h u\|^2_{L^2(\om)}+\|S(\Pi_h u-u)\|^2_{L^2(\om)}.
		\end{align*}
		In view of Proposition \ref{p:state_estimate} (applied with control $\Pi_h u$) and the continuity of control--to--state operator; using a generic constant $c$, we get that 
		\begin{align}
			(\hat{y}_h-{y},(S_h-S)\Pi_h u +S\hat{u}_h-{u}))&\leq \frac{1}{2}\|{y}-\hat{y}_h\|^2_{L^2(\om)}
			+c_A^2 h^4+ c\|\Pi_h u-u\|_{L^1(\om)} \nonumber \\
			&\leq  \frac{1}{2}\|{y}-\hat{y}_h\|^2_{L^2(\om)}
			+ch^4+ch. \label{eq:lem2_2}
		\end{align}
		The last two terms in \eqref{eq:lem2_0} can be majorized as follows:
		\begin{align}
			-\beta(g,\hat{u}_h -u)_{L^2(\om)}-\beta (\hat g, \Pi_hu -\hat{u}_h)_{L^2(\om)}=& \beta(\hat g-g,\hat{u}_h -u)_{L^2(\om)}-\beta (\hat g, \Pi_hu -{u})_{L^2(\om)}\nonumber  \\
			\leq &  \beta (\hat g-g, \hat{u}_h-u)_{L^2(\om)}+{ \beta \norm{\hat g}_{L^\infty(\om)} \norm{ u - \Pi_hu}_{L^{1}(\om)}}. \label{eq:lem2_333}
		\end{align}
		By applying Proposition \ref{p:state_estimate} and Corollary \ref{est_ele_fin3}, and inserting \eqref{eq:lem2_1}--\eqref{eq:lem2_333} in \eqref{eq:lem2_0}; using a generic constant $c$ independent of $h$, we obtain
		\begin{align}
			\alpha\norm{ u - {\hat u_h}}&_{L^2(\om)}^2 +\|{y}-\hat{y}_h\|_{L^2}^2 \nonumber 
			\\
			&\leq c\big(\norm{\phi}_{L^\infty(\om)}+\norm{u}_{L^\infty(\om)} \big) h + \epsilon {\alpha} \norm{\hat{u}_h-u}^2_{L^2(\om)}\nonumber 
			\\
			& \quad  + c\frac{\alpha}{4 \epsilon}  h +\frac{1}{2}\|{y}-\hat{y}_h\|^2_{L^2(\om)} +ch^4+ch +\norm{y-y_d}_{L^2(\om)}c_A h^2\nonumber
			\\
			& \quad+\beta c\norm{\hat g}_{L^\infty(\om)}  h + \beta (g-\hat g, u -\hat{u}_h)_{L^2(\om)}, \nonumber
		\end{align}
		by taking similar terms to the left--hand side and using $c$ as a generic constant (independent of $h$) in the higher order terms, we get
		\begin{align}
			(1-\epsilon){\alpha}\norm{ u - {\hat u_h} }_{L^2(\om)}^2 + \frac12\|{y}-\hat{y}_h\|_{L^2}^2 &\leq c h + c \frac{\alpha}{4\epsilon} h+ch^4+ch +\norm{y-y_d}_{L^2(\om)} c_A h^2 \nonumber
			\\
			& \hspace{1.8cm}  +{\beta c\norm{\hat g}_{L^\infty(\om)}  h} + \beta (g-\hat g, u -\hat{u}_h)_{L^2(\om)} \nonumber\\
			&= \hat c(\epsilon) h  + \beta (g-\hat g, u -\hat{u}_h)_{L^2(\om)} \nonumber
		\end{align}
		Therefore, by considering the constant $\hat c (\epsilon) = \frac{\alpha}{4\epsilon}c  +c +c_A\norm{y-y_d}_{L^2(\om)}  +\beta c\norm{\hat g}_{L^\infty(\om)}$, independent of $h$, we obtain the desired result.
	\end{proof}
	
	\subsection{Order of convergence}
	\begin{remark}
		We emphasize that the estimates obtained in the previous results depend on the difference $g - \hat g$. For problem \eqref{e:OCP_c}, we may apply these results taking $g=w(\bar u)$ and $\hat g = w(\bar u_h)$ and estimate the terms involving the difference $w(\bar u)-w(\bar u_h)$ taking advantage of the Lipschitz continuity of the mapping $u \mapsto w(u)$. However, more information can be acquired using  Pontryagin's maximum principle.  
	\end{remark}
	
	\begin{theorem}\label{t:main_estimate}
		Let $\bar u \in U_{ad}$ solution problem \eqref{e:OCP_c}. Under Assumptions \ref{a:local_scc} and \ref{a:bsupport},  we have the following error estimate for sufficiently small mesh--size parameter $h$: 
		\begin{align*}
			\norm{\bar u - {\bar u_h}}_{L^2(\om)} +\|\bar{y}-\bar{y}_h\|_{L^2} \leq &  \,c h^\frac{1}{2},
		\end{align*}
		for small enough and some constant $c>0$ independent of $h$. 
	\end{theorem}
	\begin{proof}
		By choosing $u= \bar u$ and $\hat u_h =\bar u_h$ as well as $g=\bar w=j(\bar u)$ and $\hat g=\bar w_h = j(\bar u_h)$ in Lemma \ref{l:aux_stable2} it follows that 
		\begin{align}\label{eq:main1}
			(1-\epsilon) {\alpha}\norm{\bar u - {\bar u_h}}_{L^2(\om)}^2 +\frac12\|\bar{y}-\bar{y}_h\|_{L^2}^2 \leq &  \,\hat c (\epsilon)h + \beta(\bar w -\bar w_h, \bar u - \bar u_h)_{L^2(\om)}. 
		\end{align}
		Notice  that $\norm{\bar g}_{L^\infty(\om)} \leq \norm{w(\bar u_h)}_{L^\infty(\om)}\leq \delta_\gamma$ due to the Lemma \ref{l:w_prop}; thus, $\hat c(\epsilon)$ does not depend on $h$.
		
		Now, let us estimate the nonnegative term $(\bar w -\bar w_h,\bar u - \bar u_h)_{L^2(\om)}$ in a manner that can be taken to the left--hand side of \eqref{eq:main1}. 
		Since  the Corollary \ref{est_ele_fin3} (applied with control $\norm{w(x)-\Pi_h w(x)}_{L^1(\om)}$), there is a positive constant $\hat c$ such that 
		\begin{align}
			\int_\om (\bar w(x) -\bar w_h(x))(\bar u(x) - \bar u_h (x))\,dx \leq \hat c h + \int_\om (\Pi_h\bar w(x) -\bar w_h(x))(\bar u(x) - \bar u_h (x))\,dx. \label{eq:diffw1}
		\end{align}
		Next, observe that $\Pi_h\bar w$ and $\bar w_h$ belong to $U_h$. Therefore, by expressing both quantities in terms of the basis of $U_h$, we have
		\begin{align}
			\int_\om (\Pi_h\bar w(x) -&\bar w_h(x))(\bar u(x) - \bar u_h (x))\,dx = \int_{\om}\sum_{T\in{\mathcal T}_h} \l(\bar w_T- \bar{w}_{h_T} \r) \chi_T(x) (\bar u(x) - \bar u_h (x))\,dx,\nonumber
		\end{align}	
		
		
		We continue our analysis of this expression by  taking into account the sets of elements  $\mathcal{T}_h^0$, $\mathcal{T}_h^\#$ and $\mathcal{T}_h^*$. First, using the same arguments of Theorem 4, we infer that $\bar w_{h_T}=0$ for all $T\in  \mathcal{T}_h^0$  and that $\sign(\bar u_{h_T}) = \sign(\bar u_T)\not =0$, for all $T \in \mathcal{T}_h^\#$ if $h$ is sufficiently small. Hence, using the Assumption \ref{a:bsupport} we obtain
		\begin{align}
			\int_\om (\Pi_h\bar w(x) -&\bar w_h(x))(\bar u(x) - \bar u_h (x))\,dx \nonumber \\
			&\le \int_{\om}\Big|\sum_{T\in{\mathcal T}^{\#}_h} \l( \frac{1}{|T|} \int_T \bar w(\xi) \,d\xi- \bar{w}_{h_T} \r) \chi_T(x)\,\Big| \left|\bar u(x) - \bar u_h (x) \right|\,dx+ch. \label{eq:diffw.1}
		\end{align}
		Since $j$ is differentiable for $T\in\mathcal{T}_h^\#$, this gives
		\begin{align}
			\Big|\sum_{T\in{\mathcal T}^{\#}_h} \Big(\frac{1}{|T|}&\int_T \bar w(\xi) \,d\xi- \bar{w}_{h_T}\Big)\chi_T(x)\Big|  
			\leq 
			\sum_{T\in{\mathcal T}^{\#}_h}\Big| \Big(\frac{1}{|T|} \int_T \bar w(\xi) \,d\xi- \bar{w}_{h_T} \Big)\chi_T(x)\Big| \nonumber \\
			\leq& \sum_{T\in{\mathcal T}^{\#}_h} \Big( \frac{1}{|T|} \int_T\Big| \bar w(\xi) -\bar w_{h_T}\Big| d\xi\Big)\chi_T(x) \nonumber 
			\\
			\leq & \sum_{T\in{\mathcal T}^{\#}_h} \Big(\frac{q}{|T|}  \int_T \Big|\l(|\bar u_{h_T}| + \frac{q-1}{\gamma}\r)^{q-1} -\l(|\bar u(\xi)| + \frac{q-1}{\gamma}\r)^{q-1}\Big|d\xi\Big) \chi_T(x) 
			\nonumber\\
			\leq & \sum_{T\in{\mathcal T}^{\#}_h} q(1-q) \Big(\frac{1}{|T|} \int_T \l(|\bar u (\xi)| + \frac{q-1}{\gamma}\r)^{q-2} |\bar u(\xi) -\bar u_{h_T} |d\xi\Big) \chi_T(x) \nonumber 
			\\
			+	&\sum_{T\in{\mathcal T}^{\#}_h}\frac{q(1-q)(2-q)}{|T|}\Big( \int_T \l(|\tilde u (\xi)|+ \frac{q-1}{\gamma}\r)^{q-3} (\bar u(\xi) -\bar u_{h_T} )^2 d\xi \Big)\chi_T(x),
			\label{eq:diffw.2}
		\end{align}
		where $\tilde u (\xi)$ lies between $\bar u(\xi)$ and $\bar u_{h_T}$. Moreover, since $|\tilde u(\xi)|>\rho +\frac{1}{\gamma}$, it follows that
		\begin{align}
			\Big|\sum_{T\in{\mathcal T}^{\#}_h} \Big(\frac{1}{|T|} \int_T \bar w(\xi) \,d\xi- &\bar{w}_{h_T}\Big)\chi_T(x)\Big|\nonumber \\
			\leq & \sum_{T\in{\mathcal T}^{\#}_h} q(1-q) \Big(\frac{1}{|T|} \int_T \l(|\bar u (\xi)| + \frac{q-1}{\gamma}\r)^{q-2} |\bar u(\xi) -\bar u_{h_T} |d\xi \Big)\chi_T(x)\nonumber \\
			&+ \sum_{T\in{\mathcal T}^{\#}_h}\frac{q(1-q)(2-q)}{|T|} \Big(\int_T \l(\frac{q}{\gamma}\r)^{q-3} (\bar u(\xi) -\bar u_{h_T} )^2 d\xi\Big)\chi_T(x)
			\nonumber \\
			=& \sum_{T\in{\mathcal T}^{\#}_h} q(1-q) \Big(\frac{1}{|T|} \int_T \l(|\bar u (\xi)| + \frac{q-1}{\gamma}\r)^{q-2} |\bar u(\xi) -\bar u_{h_T} |d\xi\Big)\chi_T(x) \nonumber  \\
			&+ c_q \sum_{T\in{\mathcal T}^{\#}_h} \Big(\frac{1}{|T|}   \int_T (\bar u(\xi) -\bar u_{h_T} )^2 d\xi\Big)\chi_T(x),
			\label{eq:diffw.3}
		\end{align}
		with $c_q = q(1-q)(2-q)\l(\frac{q}{\gamma}\r)^{q-3} $.
		Therefore, Theorem \ref{p:u_softt} and  Assumption \ref{a:local_scc} imply that
		\begin{align}
			\beta(1-q)q\Big(\bar u(x) +\frac{q-1}{\gamma}\Big)^{q-2}
			<   c_{_\#}  {\alpha},
		\end{align}
		where $c_{_\#}$ is a positive constant in $(0,1)$. Then, because of monotonicity it follows that if $ 0< \bar u(\xi) \le u_b $, therefore $\beta (1-q)q\Big(\bar  u(\xi)+\frac{q-1}{\gamma}\Big)^{q-2}<{\alpha}$. By similar arguments, we arrive to the same conclusion in the case $u_a\leq \bar u(\xi)<-\rho - \frac{1}{\gamma}<0$. Hence, using this relation in \eqref{eq:diffw.3} we obtain 
		\begin{align}
			\beta\Big| \sum_{T\in{\mathcal T}^{\#}_h} \Big(\frac{1}{|T|} \int_T \bar w(\xi) \,d\xi- \bar{w}_{h_T}\Big)\chi_T(x)
			\Big|
			\leq & {\alpha} c_{_\#} \sum_{T\in{\mathcal T}^{\#}_h} \Big(\frac{1}{|T|} \int_T |\bar u(\xi) -\bar u_{h_T} |d\xi\Big)\chi_T(x)\nonumber\\
			& + \beta c_q \sum_{T\in{\mathcal T}^{\#}_h} \Big(\frac{1}{|T|}   \int_T (\bar u(\xi) -\bar u_{h_T} )^2 d\xi\Big)\chi_T(x)   \nonumber  
			\\
			\leq & {\alpha} c_{_\#}  \Pi_h(|\bar u - \bar{u}_{h}|) + \beta c_q \Pi_h(\bar u - \bar{u}_{h})^2,
			\nonumber
		\end{align}
		which we insert into \eqref{eq:diffw.1}. Also, applying Young's inequality we get
		
		\begin{align}
			\beta \int_\om (\Pi_h\bar w(x) &-\bar w_h(x))(\bar u(x) - \bar u_h (x))\,dx
			\leq  {\alpha} c_{_\#}\int_\om \Pi_h(|\bar u - \bar{u}_{h}|)|\bar u(x) - \bar u_h (x)|\,dx\nonumber 
			\\
			&\qquad +\beta c_q \int_\om \Pi_h(\bar u - \bar{u}_{h})^2|\bar u(x) - \bar u_h (x)|\,dx+c h \nonumber 
			\\
			&
			\leq {\alpha}  c_{_\#}  \norm{\bar u-\bar u_h}^2_{L^2(\om)} + c \norm{\bar u - \bar u_h}_{L^\infty(\om)}\norm{\bar u - \bar u_h}_{L^2(\om)}^2+c h, \label{eq:bluediffw.6}
		\end{align}
		where $c$ and $C$ are positive constants. In addition, by Theorem  \ref{t:convergence} (iii), there exists  $h_2>0$ such that $\norm{\bar u - \bar u_h}_{L^\infty(\om)}\leq \epsilon\alpha/c$, for  all $h\leq h_2$. Therefore
		\begin{align}
			\beta \int_\om (\Pi_h\bar w(x) -\bar w_h(x))(\bar u(x) - \bar u_h (x))\,dx
			\leq {\alpha}  c_{_\#} & \norm{\bar u-\bar u_h}^2_{L^2(\om)} 
			\nonumber 
			\\
			&
			+\epsilon{\alpha}\norm{\bar u - \bar u_h}_{L^2(\om)}^2+c h. \label{eq:bluediffw.76}
		\end{align}
		Using \eqref{eq:main1}, \eqref{eq:diffw1} and \eqref{eq:bluediffw.76} we  estimate 
		\begin{align}\label{eq:main2}
			(1-2\epsilon) \alpha\norm{\bar u - {\bar u_h}}_{L^2(\om)}^2 +\frac12\|\bar{y}-\bar{y}_h\|_{L^2}^2 \leq &  \,\hat c(\epsilon) h + \alpha c_{_\#} \norm{\bar u-\bar u_h}^2_{L^2(\om)},
		\end{align}
		where the positive constants $\hat c (\epsilon)$ and $c$ were redefined.
		Finally, choosing $\epsilon=\frac{1-c_{_\#}}{4}$ and taking the term $c_{_\#}{\alpha} \norm{\bar u-\bar u_h}_{L^2(\om)} $ to the left--hand side and obtaining the square root, we conclude the assertion.
	\end{proof}

	\section{Numerical experiments}
	In this final section, we verify the previous error estimate through an example implemented in Matlab. The numerical solution is computed by discretizing the state and adjoint equation in the corresponding approximation spaces. Then we solve the optimality system using a semi-smooth Newton method in the spirit of \cite{merino2021}.
	
	To confirm our theoretical convergence rate, we calculate the experimental order of convergence by using a family of decreasing values of $h$  meshes. We consider the exact solution as the corresponding numerical solution computed for a thin mesh, with which we use a reference to compare the approximated solutions for each mesh. The experimental order of convergence  is calculated by:
	\begin{equation}\label{err_conv_ex}
		EOC:=\frac{\log(||\bar{u}_{h^*}-\bar{u}_{h_1}||)-\log(||\bar{u}_{h^*}-\bar{u}_{h_2}||)}{\log(h_1)-\log(h_2)}. 
	\end{equation}
	where $h_1$ and $h_2$ are two consecutive mesh sizes, and $\bar{u}_h^*$ the approximate reference solution of the problem. In on the unit square domain $\Omega=(0,1)\times(0,1)$.

	\begin{example}
		Our first example is defined on the unit square domain $\Omega=(0,1)\times(0,1)$ and consider  $A=-\Delta$ and $y_d=10e^{-5(x^2+y^2)}$ together with fixed parameters $\alpha=0.24$, $\beta=0.0002$.  The set of admissible controls is
		\[
		\mathcal{U}_{ad}=\{u\in L^2:-0.8\le u\le 0.55,\quad\text{a.a. } x\in\Omega\}.
		\]
		The regularization parameter is set to $\gamma=16000$.
		The exact solution to this problem is not known. Therefore, we compute a reference solution $\bar{u}_h^*=\bar{u}_h$  for $h=0.0006$. 
		The results are presented for different values of  $q$  in Table \ref{tab1otro}.
	\end{example}
	
	\begin{table}[h!t!b!]
		{\tiny \begin{center}
				\begin{tabular}{SSSSSSSS} \toprule
					& &{$h_1$}& {$h_2$}& {$h_3$}& {$h_4$}& {$h_5$}& {$h_6$}\\
					\midrule
					{$q$} &   {$h$}   & 0.0366 & 0.0183 & 0.0092 &  0.0046  & 0.0023&0.0011 \\
					\midrule
					0.5&	{$\| \bar{u}_h-\bar{u}_h^*\|_{L^2(\om)}$}&  0.1784  &  0.0796    & 0.0318  &  0.0155   &  0.0083  &  0.0047\\
					& {$\text{EOC}$}&{-}&1.2&    1.3 &    1.0 &    0.9 &    0.8\\
					\midrule
					0.48&	{$\| \bar{u}_h-\bar{u}_h^*\|_{L^2(\om)}$}& 0.1819  &  0.0818 &   0.0333  &  0.0165 &   0.0104 &   0.0063\\
					&{$\text{EOC}$}&{-}&1.2 &   1.3&    1.0  &   0.7 &   0.7 \\
					\midrule
					0.45&	{$\| \bar{u}_h-\bar{u}_h^*\|_{L^2(\om)}$}& 0.1859  &  0.0862    & 0.0365  &  0.0176   & 0.0117   & 0.0067 \\
					&{$\text{EOC}$}&{-}&1.1 &   1.2    & 1.1  &  0.6   &  0.8\\
					\midrule
					0.41&	{$\| \bar{u}_h-\bar{u}_h^*\|_{L^2(\om)}$}&    0.1876  &  0.0912   & 0.0369  &  0.0193 &   0.0128  &  0.0083\\
					&{$\text{EOC}$}&{-}&1.0  &    1.3    & 0.9   &  0.6    & 0.6\\
					\midrule
					0.40&	{$\| \bar{u}_h-\bar{u}_h^*\|_{L^2(\om)}$}& 0.1876  &  0.0912    &0.0369 &   0.0193   & 0.0128   & 0.0083 \\
					
					&{$\text{EOC}$}&{-}&1.0    & 1.3    & 0.9   &  0.6 &   0.6 \\
					\midrule
					0.38&	{$\| \bar{u}_h-\bar{u}_h^*\|_{L^2(\om)}$}& 0.2019 &   0.0975   & 0.0417  &  0.0259  &  0.0167  &  0.0112 \\
					&{$\text{EOC}$}&{-}&1.1 &   1.2   &  0.7 &   0.6   &  0.6\\ 
					\midrule
					0.37&	{$\| \bar{u}_h-\bar{u}_h^*\|_{L^2(\om)}$}&  0.2112 &   0.1104  &  0.0467  &  0.0276 &   0.0185  &  0.0125 \\
					&{$\text{EOC}$}&{-}&0.9   &  1.2   &  0.8 &   0.6  & 0.6\\ 
					\midrule
					0.36&	{$\| \bar{u}_h-\bar{u}_h^*\|_{L^2(\om)}$}&0.2240 &   0.1084  &  0.0465  &  0.0301  &  0.0210   & 0.0137\\
					&{$\text{EOC}$}&{-}&  1.0   &  1.2  &   0.6   &  0.5  &   0.6 \\
					\midrule
					0.34&	{$\| \bar{u}_h-\bar{u}_h^*\|_{L^2(\om)}$}&  0.2382 &   0.1118&    0.0494  &  0.0349  &  0.0235  &  0.0151\\
					&{$\text{EOC}$}&{-}& 1.1&   1.1   &  0.5  &  0.6   & 0.6 \\
					\midrule
					0.33&	{$\| \bar{u}_h-\bar{u}_h^*\|_{L^2(\om)}$}&0.2382  &  0.1118   & 0.0493&    0.0349 &   0.0234  &  0.0150\\
					&{$\text{EOC}$}&{-}&1.1  &  1.2  &  0.5   &  0.6  &  0.6 \\
					
					\midrule
					0.32&	{$\| \bar{u}_h-\bar{u}_h^*\|_{L^2(\om)}$}&0.2382 &   0.1118 &   0.0493  &  0.0349  &  0.0234  &  0.0150\\
					&{$\text{EOC}$}&{-}&    1.1 &  1.1   &  0.5   &  0.6   & 0.6 \\
					\midrule
					0.31&	{$\| \bar{u}_h-\bar{u}_h^*\|_{L^2(\om)}$}& 0.2382  &   0.1118 &   0.0493 &   0.0349  &  0.0234  &  0.0151\\
					&{$\text{EOC}$}&{-}&1.1  &  1.2  &  0.5   &  0.6  &  0.6 \\ 
					\bottomrule
				\end{tabular}
				\caption{Computed error $\| \bar{u}_h-\bar{u}_h^*\|_{L^2(\om)}$ and experimental order of convergence for different values of the exponent $q$}
				\label{tab1otro}
		\end{center}}
	\end{table}
	
	We notice the tendency of an order $\approx \frac{1}{2}$ in each for different values of $q$ as $h$ decreases, see Figure \ref{fig_example1}. 
	\begin{figure}[h]
		\begin{subfigure}{.31\textwidth}
			\includegraphics[scale=0.25]{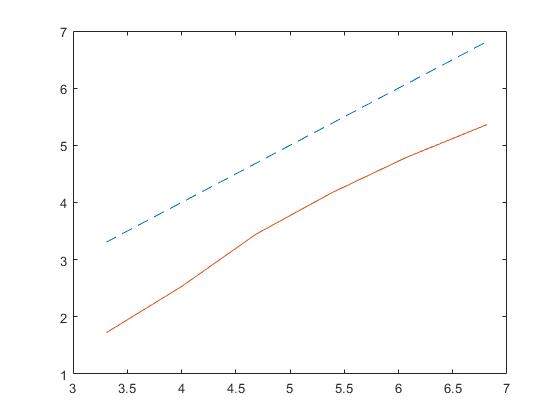}
			\caption{$q=0.5$.}
		\end{subfigure}%
		\begin{subfigure}{0.31\textwidth}
			\includegraphics[scale=0.25]{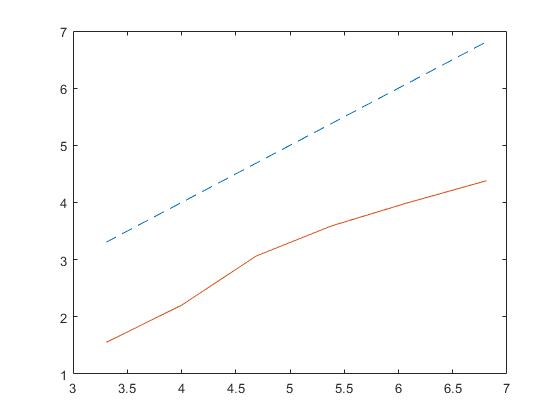}
			\caption{$q=0.38$.}
		\end{subfigure}%
		\begin{subfigure}{.31\textwidth}
			\includegraphics[scale=0.25]{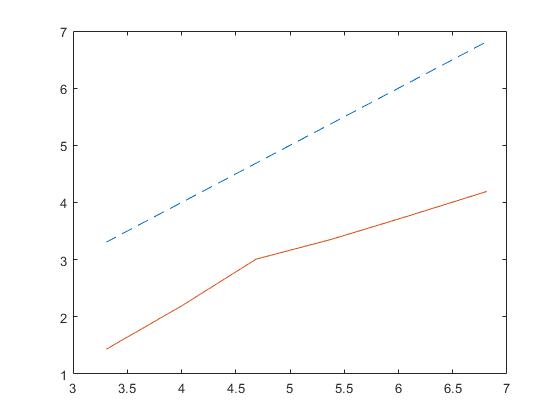}
			\caption{$q=0.31$.}
		\end{subfigure}%
		\caption{$-\log(h)$ versus $-\log(\|\bar u^*_h-\bar u_h\|)$ (solid line) compared with $-\frac{1}{2}\log(h)$ (dashed line)  varying the fractional exponent $q$}.
		\label{fig_example1}
	\end{figure}
	The plots of approximated optimal control for $h=0.0046$ are shown in the Figure \ref{fig_example2}, for different values of $q$. 

%
				\begin{figure}
						\begin{subfigure}{.3\textwidth}
								\includegraphics{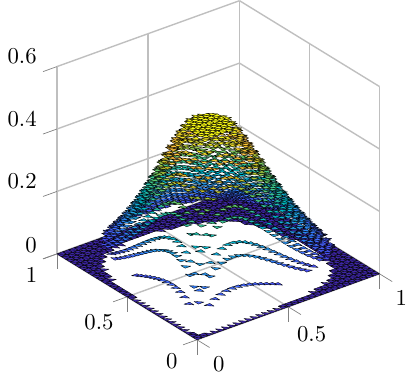}
								\caption{$q=0.5$.}
							\end{subfigure}%
						\begin{subfigure}{.3\textwidth}
								\includegraphics{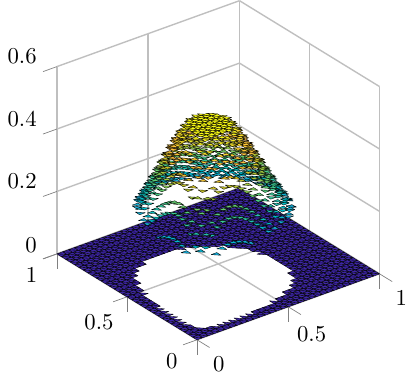}
								\caption{$q=0.38$.}
							\end{subfigure}%
						\begin{subfigure}{.3\textwidth}
								\includegraphics{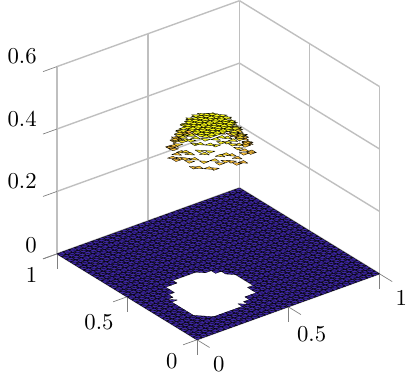}
								\caption{$q=0.31$.}
							\end{subfigure}%
						\caption{$\bar u_h$ computed at $h=0.0046$ for different values of the fractional exponent $q$}
						\label{fig_example2}
					\end{figure}


	\pagebreak
	\section{Appendix}
	In this appendix, we present an extension of the results obtained in  \cite[Section 4]{DLReMeVe}  for a quasi-interpolation operator given by 
	\begin{align}\label{general_operator_q}
		\Pi_h u:=\sum_{i}\pi_i(u) \phi_i,\quad\text{with}\quad \pi_i(u)=\frac{\int_{\omega_i}u  \phi_i}{\int_{\omega_i}\phi_i}.
	\end{align}
	
	where  $\phi_i$,  for $i=1,\dots, n$, denote the ansatz functions, and $\omega_i=\text{supp } \phi_i$.
	
	We can notice the quasi-interpolation operator given in the Definition in \ref{d:smooth_BV} is a particular case of \eqref{general_operator_q}. Therefore, the results obtained in this appendix are used to prove the Lemma \ref{est_ele_fin2} and Corollary  \ref{est_ele_fin3}.
	\begin{lemma}\label{lemma1-JC}
		For each $i \in \{1, \dots, n\},$ there is a constant $c$ which may depend on diam $\omega_i$ such that
		\[
		\norm{u-\pi_i(u)}_{L^1(\omega_i)}\le c
		\norm{\nabla u}_{L^s(\omega_i)}, \quad\forall u\in W^{1,s}(\omega_i),
		\]
		for all $\frac{2d}{d+2}\le s<\infty$.
	\end{lemma}
	\begin{proof}
		Since $L^2(\om)\hookrightarrow L^1(\om)$, the result directly follows by using  \cite[Lemma 4.1]{DLReMeVe}.\qedhere
	\end{proof}
	
	\begin{lemma}\label{lemma2-JC}
		There is a constant $c$ which is independent of $h$  such that
		\[
		\norm{u-\pi_i(u)}_{L^1(\omega_i)}\le c\, h^{d\left(1-\frac{1}{s}\right)+1}
		\norm{\nabla u}_{L^s(\omega_i)}, \quad\forall u\in W^{1,s}(\omega_i),
		\]
		for all $i\in \{1,\dots,n\}$ and all $\frac{2d}{d+2}\le s<\infty$.
	\end{lemma}
	\begin{proof}
		The proof of this theorem follows the same ideas of \cite[Lemma 4.2]{DLReMeVe}. We consider an arbitrary patch $\omega_i$  consisting  of the cells $T_j^{(i)}$, for $j = 1, \dots, M_i$. Then, for each $\omega_i$ we  associate a surface $\hat \omega_i$  whose vertices lie on  the unit ball in $\reals^d$, and every  $\hat \omega_i$ consists of $M_i$ congruent cells $\hat T_j^{(i)}$. Thus, we can define the function
		\begin{align*}
			\hat\pi_i(v)=\frac{\int_{\hat \omega_i}\hat\phi_i v \, d\hat x}{\int_{\hat\omega_i}\hat\phi_i \, d\hat x}=\frac{\int_{\hat \omega_i}(\phi_i\circ F_i) v \, d\hat x}{\int_{\hat\omega_i}\phi_i\circ F_i\, d\hat x},
		\end{align*}
		where $F_i$ denotes the bi-Lipschitz transformation from $\hat\omega_i$ to $\omega_i$. Therefore, we get the next estimation
		\begin{align}
			\norm{u-\pi_i(u)}_{L^1(\omega_i)}&=\sum_{j=1}^{M_i}\frac{|T_j^{(i)}|}{|\hat T_j^{(i)}|}\int_{\hat T_j^{(i)}}|u(F_j^{(i)}\hat x)-\pi_i(u)|d\hat x \nonumber
			\\
			&\leq c\, h^d \int_{\hat\omega_j}|u\circ F_i-\hat\pi_i(u\circ F_i)|d\hat x
			\nonumber
			\\
			&\leq c\, h^d \int_{\hat\omega_j}\left|\int_{\hat\omega_j} (\phi_h\circ F_i) (\hat y) (u\circ F_i(\hat x)-u\circ F_i(\hat y))d\hat y\right|d\hat x
			\nonumber
			\\
			&\leq c\, h^d \int_{\hat\omega_j}\norm{\phi_h\circ F_i}_{L^r(\omega_i)}\left(\int_{\hat\omega_j}|\nabla_{\hat x}(u\circ F_i)(\hat z) (\hat x-\hat y)|^sd\hat y \right)^\frac{1}{s}d\hat x
			\nonumber
			\\
			&\leq c\, h^d \,\norm{\phi_h\circ F_i}_{L^r(\omega_i)}\left(\int_{\hat\omega_j}|\nabla_{\hat x}(u\circ F_i)|^s d\hat x\right)^\frac{1}{s}
			\nonumber
			\\
			&\leq c\, h^d \sum_{j=1}^{M_i}\frac{|\hat T_j^{(i)}|}{|T_j^{(i)}|}\left(\int_{ T_j}|\nabla_{x}u|^s |\frac{\partial x}{\partial \hat x} |^s d x \right)^\frac{1}{s}
			\nonumber
			\\
			&\leq c\, h^{d\left(1-\frac{1}{s}\right)+1}\norm{\nabla u}_{L^s(\omega_i)}.\nonumber
		\end{align}
	\end{proof}
	\begin{lemma}\label{l:L1-quasiint}
		There is a constant $c$  independent of $h$,  such that
		\[
		\norm{u-\Pi_h(u)}_{L^1(\om)}\le c\, h^{d\left(1-\frac{1}{s}\right)+1}
		\norm{\nabla u}_{L^s(\om)}, \quad\forall u\in W^{1,s}(\om),
		\]
		with $\frac{2d}{d+2}\le s\le 2$.
	\end{lemma}
	\begin{proof}
		Following the proof of \cite[Lemma 4.3]{DLReMeVe} we can estimate
		\begin{align}
			\norm{u-\pi_i(u)}_{L^1(\om)}
			&=\int_{\om}\left|u\sum_{i=1}^{n}\phi_i-\sum_{i=1}^{n}\pi_i(u)\phi_i\right|dx \nonumber
			\\
			&\le\sum_{i=1}^{n}\int_{\omega_i}|u-\pi_i(u)||\phi_i|d\hat x \nonumber
			\\
			&\leq c\, h^{d\left(1-\frac{1}{s}\right)+1}\sum_{i=1}^{n}\norm{\nabla u}_{L^s(\omega_i)}
			\nonumber
			\\
			&\leq c\, h^{d\left(1-\frac{1}{s}\right)+1}\left(\sum_{i=1}^{n}\norm{\nabla u}^s_{L^s(\omega_i)}\right)^\frac{1}{s}
			\nonumber
			\\
			&\leq c\, h^{d\left(1-\frac{1}{s}\right)+1}\norm{\nabla u}_{L^s(\om)}.
			\nonumber
		\end{align}
	\end{proof}

\end{document}